\documentclass[oneside, 11pt]{amsart} 
\usepackage{amsmath,amsthm,amssymb,epic}
\usepackage{eqlist,eqparbox}
\usepackage{amsfonts}
\usepackage{latexsym}
\usepackage[2emode]{psfrag}
\usepackage{amsthm}
\usepackage{amsmath}
\usepackage[all]{xy}

\newcommand{\Title}[1]{\bigskip\bigskip\centerline{\bf #1}\bigskip}
\newcommand{\Author}[1]{\medskip\centerline{ \it #1}}

\newcommand{\Affiliation}[1]{\medskip\centerline{#1}}
\newcommand{\Email}[1]{\medskip\centerline{#1}\bigskip}

\begin{document}

\newcommand{\N}{\mbox {$\mathbb N $}}
\newcommand{\Z}{\mbox {$\mathbb Z $}}
\newcommand{\Q}{\mbox {$\mathbb Q $}}
\newcommand{\R}{\mbox {$\mathbb R $}}
\newcommand{\lo }{\longrightarrow }
\newcommand{\ul}{\underleftarrow }
\newcommand{\rl}{\underrightarrow }
\newcommand{\rs }{\rightsquigarrow }
\newcommand{\ra }{\rightarrow }
\newcommand{\dd }{\rightsquigarrow }
\newcommand{\ol }{\overline }
\newcommand{\la }{\langle }
\newcommand{\tr }{\triangle }
\newcommand{\xr }{\xrightarrow }
\newcommand{\de }{\delta }
\newcommand{\pa }{\partial }
\newcommand{\LR }{\Longleftrightarrow }
\newcommand{\Ri }{\Rightarrow }
\newcommand{\va }{\varphi }
\newcommand{\Den}{{\rm Den}\,}
\newcommand{\Ker}{{\rm Ker}\,}
\newcommand{\Reg}{{\rm Reg}\,}
\newcommand{\Fix}{{\rm Fix}\,}
\newcommand{\Img}{{\rm Im}\,}
\newcommand{\Id}{{\rm Id}\,}

\newtheorem{theorem}{Theorem}[section]
\newtheorem{lemma}[theorem]{Lemma}
\newtheorem{proposition}[theorem]{Proposition}
\newtheorem{corollary}[theorem]{Corollary}
\newtheorem{definition}[theorem]{Definition}
\newtheorem{example}[theorem]{Example}
\newtheorem{examples}[theorem]{Examples}
\newtheorem{xca}[theorem]{Exercise}
\theoremstyle{remark}
\newtheorem{remark}[theorem]{Remark}
\numberwithin{equation}{section}

\def\leftmark{L.C. Ciungu}

\Title{MONADIC PSEUDO BE-ALGEBRAS} 
\title[Monadic pseudo BE-algebras]{}
                                                                           
\Author{\textbf{LAVINIA CORINA CIUNGU}}
\Affiliation{Department of Mathematics} 
\Affiliation{University of Iowa}
\Affiliation{14 MacLean Hall, Iowa City, Iowa 52242-1419, USA}
\Email{lavinia-ciungu@uiowa.edu}

\begin{abstract} 
In this paper we define the monadic pseudo BE-algebras and investigate their properties. 
We prove that the existential and universal quantifiers of a monadic pseudo BE-algebra form a residuated pair. 
Special properties are studied for the particular case of monadic bounded commutative pseudo BE-algebras. 
Monadic classes of pseudo BE-algebras are investigated and it is proved that the quantifiers on bounded commutative pseudo BE-algebras are also quantifiers on the corresponding pseudo MV-algebras. 
The monadic deductive systems and monadic congruences of monadic pseudo BE-algebras are defined and 
their properties are studied. 
It is proved that, in the case of a monadic distributive commutative pseudo BE-algebra there is a one-to-one 
correspondence between monadic congruences and monadic deductive systems, and the monadic quotient 
pseudo BE-algebra algebra is also defined. \\ 

\textbf{AMS classification (2000):} 03G25, 06F35, 003B52
\end{abstract}

\maketitle

\section{Introduction}

Many information processing branches are based on the non-classical logics and deal with
uncertainty information (fuzziness, randomness, vagueness, etc.).
Developing algebraic models for non-commutative multiple-valued logics is a central topic in the study of 
fuzzy systems. Pseudo BE-algebra is such an algebraic structure.  
Starting from the systems of positive implicational calculus, weak systems of positive implicational calculus and BCI and BCK systems, in 1966 Y. Imai and K. Is$\rm\acute{e}$ki introduced the BCK-algebras (\cite{Imai}). 
BCK-algebras are also used in a dual form, with an implication $\to$ and with one constant element $1$, that is the greatest element. 
BE-algebras have been defined in \cite{Kim1} as a generalization of BCK-algebras, 
Dual BCK-algebras were defined in \cite{Kim2} and it was proved in \cite{Wal1} that any dual BCK-algebra is a BE-algebra and any commutative BE-algebra is a dual BCK-algebra. 
Pseudo BCK-algebras were introduced by G. Georgescu and A. Iorgulescu in \cite{Geo15} as algebras 
with "two differences", a left- and right-difference, and with a constant element $0$ as the least element. Nowadays pseudo BCK-algebras are used in a dual form, with two implications, $\to$ and $\rightsquigarrow$ and with one constant element $1$, that is the greatest element. 
Thus such pseudo BCK-algebras are in the "negative cone" and are also called "left-ones".  
Pseudo BE-algebras were introduced in \cite{Bor2} as generalizations of BE-algebras and properties of 
these structures have recently been studied in \cite{Bor4} and \cite{Rez1}. 
The commutative pseudo BE-algebras were defined and studied in \cite{Ciu33}, and it was proved that the class of commutative BE-algebras is term equivalent to the class of commutative pseudo BCK-algebras.  \\
The concept of \emph{existential quantifier} on a Boolean algebra $A$ was introduced by Halmos in \cite{Halmos} 
as a map $\exists:A\longrightarrow A$, such that: 
$(i)$ $\exists \perp=\perp$ ($\perp$ is the smallest element of $A$), $(ii)$ $x\le \exists x$, 
$(iii)$ $\exists(x\wedge \exists y)=\exists x\wedge \exists y$, for all $x, y\in A$. 
The pair $(A,\exists)$ was called \emph{Boolean monadic algebra} and the theory of monadic Boolean algebras is an algebraic treatment of the logic of propositional functions of one argument, with Boolean operations and a single
(existential) quantifier. The properties of new concept have been studied by many researchers.
Monadic MV-algebras were introduced and studied in \cite{Rut1} as an algebraic model of the predicate calculus 
of the  \L ukasiewicz infinite valued logic in which only a single individual variable occurs. 
Recently, the theory of monadic MV-algebras has been developed in \cite{Bell1}, \cite{Dino1}, \cite{Geo20}.
Monadic operators were defined and investigated on various algebras of fuzzy logic: 
Heyting algebras (\cite{Bez1}), basic algebras (\cite{Cha1}), GMV-algebras (\cite{Rac1}), involutive pseudo 
BCK-algebras (\cite{Ior16}), bounded commutative R$\ell$-monoids (\cite{Rac4}), bounded residuated lattices (\cite{Rac3}), residuated lattices (\cite{Kondo1}), 
BE-algebras (\cite{Zah1}), Wajsberg hoops (\cite{Cim1}), BL-algebras (\cite{Cas1}), bounded hoops (\cite{Wang1}), 
pseudo equality algebras (\cite{Ghor1}), pseudo BCI-algebras (\cite{Xin1}), NM-algebras (\cite{Wang2}). \\

In this paper we define the monadic pseudo BE-algebras and investigate their properties. 
We prove that, if a monadic pseudo BE-algebra has the transitivity property, then the existential and universal quantifiers form a residuated pair. 
Special properties are studied for the particular case of monadic bounded commutative pseudo BE-algebras and  
a method to construct a monadic operator $(\exists,\forall)$ on these structures is given. 
For the monadic bounded commutative pseudo BE-algebras it is proved that there is a one-to-one correspondence 
between universal and existential quantifiers. 
Monadic classes of pseudo BE-algebras are investigated: pseudo BCK-algebras, pseudo BCK-semilattices, 
pseudo BCK(pP)-algebras, involutive pseudo BCK-algebras, pseudo-hoops. 
It is proved that the quantifiers on bounded commutative pseudo BE-algebras are also quantifiers on the 
corresponding pseudo MV-algebras. 
The monadic deductive systems and monadic congruences of monadic pseudo BE-algebras are defined and 
their properties are studied. 
It is proved that, in the case of a monadic distributive commutative pseudo BE-algebra there is a one-to-one 
correspondence between monadic congruences and monadic deductive systems, and the monadic quotient 
pseudo BE-algebra algebra is also defined. 
Similar results are proved in the particular case of pseudo BCK-meet-semilattices.

$\vspace*{5mm}$

\section{Preliminaries}

In this section we recall some basic notions and results regarding pseudo BE-algebras and pseudo BCK-algebras,  
as well as their classes used in the paper: pseudo BCK(pP) algebras, pseudo-hoops, pseudo MV-algebras. 

\begin{definition} \label{psBE-10} $\rm($\cite{Kuhr5}$\rm)$
A structure $(A,\rightarrow,\rightsquigarrow,1)$ of the type $(2,2,0)$ is a \emph{pseudo BCK-algebra} 
(more precisely, \emph{reversed left-pseudo BCK-algebra}) iff it satisfies the following identities and 
quasi-identity, for all $x, y, z \in A$:\\ 
$(psBCK_1)$ $(x \rightarrow y) \rightsquigarrow [(y \rightarrow z) \rightsquigarrow (x \rightarrow z)]=1;$ \\
$(psBCK_2)$ $(x \rightsquigarrow y) \rightarrow [(y \rightsquigarrow z) \rightarrow (x \rightsquigarrow z)]=1;$ \\
$(psBCK_3)$ $1 \rightarrow x = x;$ \\
$(psBCK_4)$ $1 \rightsquigarrow x = x;$ \\
$(psBCK_5)$ $x \rightarrow 1 = 1;$ \\
$(psBCK_6)$ $(x\rightarrow y =1$ and $y \rightarrow x=1)$ implies $x=y$. 
\end{definition}

The partial order $\le$ is defined by $x \le y$ iff $x \ra y =1$ (iff $x \rs y =1$). 
If the poset $(A, \le)$ is a meet-semilattice then ${\mathcal A}$ is called a \emph{pseudo BCK-meet-semilattice} 
and we denote it by $\mathcal{A}=(A, \wedge, \ra, \rs, 1)$.       
If the poset $(A, \le)$ is a join-semilattice then ${\mathcal A}$ is called a \emph{pseudo BCK-join-semilattice} 
and we denote it by $\mathcal{A}=(A, \vee, \ra, \rs, 1)$.
If $(A,\leq)$ is a lattice then we will say that ${\mathcal A}$ is a \emph{pseudo BCK-lattice} and 
it is denoted by $\mathcal{A}=(A, \wedge, \vee, \ra, \rs, 1)$. \\
For details regarding pseudo BCK-algebras we refer the reader to \cite{Ior14}, \cite{Ior15}, \cite{Kuhr6}, 
\cite{Ciu2}, \cite{Kuhr2}. \\
If there is an element $0$ of a pseudo BCK-algebra $(A, \rightarrow, \rightsquigarrow, 1)$, such that $0\le x$ 
(i.e. $0\rightarrow x=0\rightsquigarrow x=1$), for all $x\in A$, then the pseudo BCK-algebra is said to be  
\emph{bounded} and it is denoted by $(A, \rightarrow, \rightsquigarrow, 0, 1)$. 
In a bounded pseudo BCK-algebra $(X, \rightarrow, \rightsquigarrow, 0, 1)$ we define two negations: \\
$\hspace*{3cm}$ $x^{-}:= x\rightarrow 0$, $x^{\sim}:= x\rightsquigarrow 0$, \\
for all $x\in X$. 

\begin{proposition} \label{psBE-20} $\rm($\cite{Ior1}$\rm)$ Let $(A,\rightarrow,\rightsquigarrow,1)$ be a pseudo BCK-algebra. 
Then the following hold for all $x, y, z\in A$: \\
$(1)$ $x\rightarrow (y\rightsquigarrow z)=y\rightsquigarrow (x\rightarrow z);$ \\
$(2)$ $x \leq y$ implies $y \rightarrow z \leq x \rightarrow z$ and
      $y \rightsquigarrow z \leq x \rightsquigarrow z;$ \\
$(3)$ $x \leq y$ implies $z \rightarrow x \leq z \rightarrow y$ and
      $z\rightsquigarrow x \leq z \rightsquigarrow y;$ \\
$(4)$ $x\rightarrow y\le (z\rightarrow x)\rightarrow (z\rightarrow y)$ and 
      $x\rightsquigarrow y\le (z\rightsquigarrow x)\rightsquigarrow (z\rightsquigarrow y);$ \\
$(5)$ $x\le (x\ra y)\rs y$ and $x\le (x\rs y)\ra y$.                       
\end{proposition}

\begin{proposition} \label{psBE-25} $\rm($\cite{Ior1}$\rm)$ Let $(A,\ra,\rs,0,1)$ be a bounded pseudo BCK-algebra. 
Then the following hold for all $x, y\in A$: \\
$(1)$ $1^{-}=1^{\sim}=0$, $0^{-}=0^{\sim}=1;$ \\
$(2)$ $x\le x^{-\sim}$, $x\le x^{\sim-};$ \\
$(3)$ $x\le y$ implis $y^{-}\le x^{-}$ and $y^{\sim}\le x^{\sim};$ \\
$(4)$ $x^{-\sim-}=x^{-}$, $x^{\sim-\sim}=x^{\sim};$ \\
$(5)$ $x \rightarrow y^{-\sim}=y^- \rightsquigarrow x^{-} = x^{-\sim} \rightarrow y^{-\sim};$ \\
$(6)$ $x \rightsquigarrow y^{\sim-}=y^{\sim} \rightarrow x^{\sim} = x^{\sim-} \rightsquigarrow y^{\sim-}$. 
\end{proposition}             

A pseudo BCK-algebra with the \emph{(pP) condition} (i.e. with the  \emph{pseudo-product} condition) or
a \emph{pseudo BCK(pP)-algebra} for short, is a pseudo BCK-algebra $(A,\leq,\ra,\rs,1)$ satisfying the condition:\\
(pP) For all $x, y \in A$, $x\odot y$ exists, where  \\
$\hspace*{2cm}$ $x\odot y=\min\{z \mid x \leq y \ra z\}=\min\{z \mid y \leq x \rs z\}$. \\
In any bounded pseudo BCK(pP)-algebra $A$ the following hold, for all $x, y\in A$ (see \cite{Ior1}): 
$x\odot 0=0\odot x=0$, $x^{-}\odot x=x\odot x^{\sim}=0$, $x\le y^{-}$ iff $y\le x^{\sim}$. 
     
\begin{proposition} \label{psBE-25-10} $\rm($\cite{Ciu2}$\rm)$ Let $(A,\odot,\ra,\rs,1)$ be a pseudo BCK(pP)-algebra. 
Then the following hold for all $x, y, z\in A$: \\
$(1)$ $x\odot y\le x,y;$ \\
$(2)$ $(x\ra y)\odot x\le x,y$ and $x\odot (x\rs y)\le x,y;$ \\
$(3)$ $x\le y$ implies $x\odot z\le y\odot z$ and $z\odot x\le z\odot y;$ \\
$(4)$ $x\ra y\le x\odot z\ra y\odot z$ and $x\rs y\le z\odot x\ra z\odot y;$ \\
$(5)$ $x\ra (y\ra z)=x\odot y\ra z$ and $x\rs (y\rs z)=y\odot x\rs z$.    
\end{proposition}

Pseudo BE-algebras were introduced in \cite{Bor2} as generalizations of BE-algebras and properties of these 
structures have recently been studied in \cite{Bor4} and \cite{Rez1}. 

\begin{definition} \label{psBE-30} $\rm($\cite{Bor2}$\rm)$
A \emph{pseudo BE-algebra} is an algebra $(A, \rightarrow, \rightsquigarrow, 1)$ of the type $(2, 2, 0)$ 
such that the following axioms are fulfilled for all $x, y, z\in A$: \\
$(psBE_1)$ $x\rightarrow x=x\rightsquigarrow x=1,$ \\
$(psBE_2)$ $x\rightarrow 1=x\rightsquigarrow 1=1,$ \\
$(psBE_3)$ $1\rightarrow x=1\rightsquigarrow x=x,$ \\
$(psBE_4)$ $x\rightarrow (y\rightsquigarrow z)=y\rightsquigarrow (x\rightarrow z),$ \\
$(psBE_5)$ $x\rightarrow y=1$ iff $x\rightsquigarrow y=1$. 
\end{definition}

\begin{proposition} \label{psBE-40} $\rm($\cite{Bor2}$\rm)$ Let $(A,\rightarrow,\rightsquigarrow,1)$ be a pseudo BE-algebra. 
Then the following hold for all $x, y, z\in A$: \\
$(1)$ $x\rightsquigarrow (y\rightarrow z)=y\rightarrow (x\rightsquigarrow z);$ \\
$(2)$ $x\rightarrow (y\rightsquigarrow x)=1$ and $x\rightsquigarrow (y\rightarrow x)=1;$ \\
$(3)$ $x\rightarrow (y\rightarrow x)=1$ and $x\rightsquigarrow (y\rightsquigarrow x)=1;$ \\
$(4)$ $x\rightarrow ((x\rightarrow y)\rightsquigarrow y)=1$ and 
      $x\rightsquigarrow ((x\rightsquigarrow y)\rightarrow y)=1$.  
\end{proposition}

We will refer to $(A,\ra,\rs,1)$ by its universe $A$. 
In a pseudo BE-algebra $A$ one can introduce a binary relation ``$\le$" by $x \le y$ if and only if $x\ra y=1$  
( if and only if $x\rs y=1$), for all $x, y\in A$. 
It was proved in \cite{Ciu33} that any pseudo BCK-algebra is a pseudo BE-algebra.
A pseudo BE-algebra $A$ is said to be \emph{distributive} (\cite{Bor4}) if it satisfies only one of the following conditions, for all $x, y, z\in A:$ \\
$(i)$ $x\ra (y\rs z)=(x\ra y)\rs (x\ra z);$ \\
$(ii)$ $x\rs (y\ra z)=(x\rs y)\ra (x\rs z)$. \\
In this paper $A$ is a distributive pseudo BE-algebra if it satisfies condition $(i)$. 
Note that, if a pseudo BE-algebra $A$ satisfies the both conditions $(i)$ and $(ii)$, then $A$ is a BE-algebra 
(see \cite{Bor4}). 

\begin{definition} \label{psBE-90} A pseudo BE-algebra with $(A)$ condition or a pseudo BE(A)-algebra for short, 
is a pseudo BE-algebra $(A, \ra, \rs, 1)$ such that the operations $\ra$, $\rs$ are antitone in the first variable, that is the following condition is satisfied:\\
$(A)$ if $x, y\in A$ such that $x\le y$, then: \\
$\hspace*{3cm}$ $y\ra z\le x\ra z$ and $y\rs z\le x\rs z,$ \\
for all $z\in A$. \\
A pseudo BE-algebra with $(M)$ condition or a pseudo BE(M)-algebra for short, is a pseudo BE-algebra 
$(A, \ra, \rs, 1)$ such that the operations $\ra$, $\rs$ are monotone in the first variable, 
that is the following condition is satisfied:\\
$(M)$ if $x, y\in A$ such that $x\le y$, then: \\
$\hspace*{3cm}$ $z\ra x\le z\ra y$ and $z\rs x\le z\rs y,$ \\
for all $z\in A$. \\
A pseudo BE-algebra with $(T)$ condition or a pseudo BE(T)-algebra for short, is a pseudo BE-algebra 
$(A, \ra, \rs, 1)$ such that the following condition is satisfied:\\
$(T)$ if $x\le y$ and $y\le z$, then $x\le z$ for all $x, y, z\in A$ (``$\le$" is transitive).\\
Any pseudo BCK-algebra is a pseudo BE(A)-algebra, a pseudo BE(M)-algebra and a pseudo BE(T)-algebra.
\end{definition} 

\begin{remark} \label{psBE-100}
If $A$ is a pseudo BE(A)-algebra, then the relation ``$\le$" is transitive. 
Indeed, let $x, y, z\in A$ such that $x\le y$ and $y\le z$. It follows that $y\ra z\le x\ra z$, 
so $1\le x\ra z$, that is $x\ra z=1$. Hence $x\le z$, that is the relation ``$\le$" is transitive. 
With other words, any pseudo BE(A)-algebra is a pseudo BE(T)-algebra. 
\end{remark}

\begin{definition} \label{psBE-110} $\rm($\cite{Geo16}$\rm)$
A \emph{pseudo-hoop} is an algebra $(A,\odot, \rightarrow, \rightsquigarrow,1)$ of the type $(2,2,2,0)$ such that, for all $x, y, z \in A:$\\
$(psH_1)$ $x\odot 1=1\odot x=x;$\\
$(psH_2)$ $x\rightarrow x=x\rightsquigarrow x=1;$\\
$(psH_3)$ $(x\odot y) \rightarrow z=x\rightarrow (y\rightarrow z);$\\
$(psH_4)$ $(x\odot y) \rightsquigarrow z=y\rightsquigarrow (x\rightsquigarrow z);$\\
$(psH_5)$ $(x\rightarrow y)\odot x=(y\rightarrow x)\odot y=x\odot(x\rightsquigarrow y)=y\odot(y\rightsquigarrow x)$.
\end{definition}

The non-commutative generalizations of MV-algebras called pseudo MV-algebras were introduced by G. Georgescu 
and A. Iorgulescu in \cite{Geo2} and independently by J. Rach{\accent23u}nek (\cite{Rac2}) under the name of generalized MV-algebras. 

\begin{definition} \label{psBE-120} $\rm($\cite{Geo2}$\rm)$
A pseudo MV-algebra is a structure $(A, \oplus, \odot, ^{-}, ^{\sim}, 0, 1)$ of type $(2,2,1,1,0,0)$ such that the following axioms hold for all $x, y, z \in A:$\\
$(psMV_1)$ $x\oplus (y \oplus z) = (x \oplus y) \oplus z;$\\
$(psMV_2)$ $x \oplus 0 = 0 \oplus x = x;$\\
$(psMV_3)$ $x \oplus 1 = 1 \oplus x =1;$ \\
$(psMV_4)$ $1^{-}=0$, $1^{\sim}=0;$\\
$(psMV_5)$ $(x^{-} \oplus y^{-})^{\sim}=(x^{\sim} \oplus y^{\sim})^{-};$\\
$(psMV_6)$ $x \oplus (x^{\sim} \odot y)=y \oplus (y^{\sim} \odot x) = (x \odot y^{-}) \oplus y =
            (y \odot x^{-}) \oplus x;$\\
$(psMV_7)$ $x \odot (x^{-} \oplus y) = (x \oplus y^{\sim}) \odot y;$\\
$(psMV_8)$ $(x^{-})^{\sim}=x,$\\
where $x \odot y:= (y^{-} \oplus x^{-})^{\sim}$. 
\end{definition}

We define $x \leq y$ iff $x^{-}\oplus y=1$ and ``$\leq$" defines an order relation on $A$. 
Moreover, $(A,\le)$ is a distributive lattice with the lattice operations defined as below: \\
$\hspace*{2cm}$ $x \vee y:= x \oplus x^{\sim} \odot y=y \oplus y^{\sim} \odot x = 
                 x \odot y^{-} \oplus y = y \odot x^{-} \oplus x,$ \\
$\hspace*{2cm}$ $x \wedge y:= x \odot (x^{-} \oplus y) = y \odot (y^{-} \oplus x)=
                (x \oplus y^{\sim}) \odot y =(y \oplus x^{\sim}) \odot x.$ \\

$\vspace*{5mm}$

\section{Monadic pseudo BE-algebras}

In this section we introduce the notion of a monadic pseudo BE-algebra and study their properties. 
Given a monadic pseudo BE-algebra $(A,\exists,\forall)$, we show that $\exists$ is a closure operator, while 
$\forall$ is an interior operator on $A$. It is proved that, if $\forall$ is surjective, then $\forall=Id_A$. 
Moreover, if $A$ satisfies the transitivity property, then $(\exists,\forall)$ is a residuated pair, that is 
$\exists x\le y$ iff $x\le \forall y$. Some properties of a monadic bounded pseudo BE-algebra are also proved. 
We define the set $A_{\exists\forall}$ of fixed elements of $(A,\exists,\forall)$, and we prove that, if 
$A_{\exists\forall}=A$, then $\forall =\exists=Id_A$. 

\begin{definition} \label{mpsBE-10} Let $(A,\ra,\rs,1,\exists,\forall)$ be an algebra of type $(2,2,0,1,1)$ such that 
$(A,\ra,\rs,1)$ is a pseudo BE-algebra. Then $(A,\exists,\forall)$ is called a \emph{monadic pseudo BE-algebra} 
if the following conditions are satisfied, for all $x, y\in A:$ \\
$(M_1)$ $x\ra \exists x=x\rs \exists x=1;$ \\
$(M_2)$ $\forall x\ra x=\forall x\rs x=1;$ \\
$(M_3)$ $\forall(x\ra \exists y)=\exists x\ra \exists y$, $\forall(x\rs \exists y)=\exists x\rs \exists y;$ \\
$(M_4)$ $\forall(\exists x\ra y)=\exists x\ra \forall y$, $\forall(\exists x\rs y)=\exists x\rs \forall y;$ \\
$(M_5)$ $\exists\forall x=\forall x$.  
\end{definition}

We use the notations $\exists x$ and $\forall x$ instead of $\exists (x)$ and $\forall (x)$, respectively. 
The unary operators $\exists: A\longrightarrow A$ and $\forall: A\longrightarrow A$ are called the  
\emph{existential quantifier} and \emph{universal quantifier} on $A$, respectively. 
The pair $(\exists, \forall)$ is said to be a \emph{monadic operator} on $A$. 
Denote by $\mathcal{MOP}(A)$ the set of all monadic operators on $A$. 

\begin{example} \label{mpsBE-10-10}
Obviously, if $A$ is a pseudo BE-algebra, then $(A,Id_A,Id_A)$ is a monadic pseudo BE-algebra, where 
$Id_A:A\longrightarrow A$, $Id_A(x)=x$, for all $x\in A$.
Hence, any pseudo BE-algebra is a monadic pseudo BE-algebra, that is $\mathcal{MOP}(A)\ne \emptyset$. 
\end{example}

\begin{example} \label{mpsBE-10-20}  
Consider the set $A=\{1,a,b,c\}$ and the operations $\ra,\rs$ given by the following tables:
\[
\hspace{10mm}
\begin{array}{c|ccccc}
\ra & 1 & a & b & c \\ \hline
1 & 1 & a & b & c \\ 
a & 1 & 1 & 1 & a \\ 
b & 1 & 1 & 1 & 1 \\ 
c & 1 & 1 & 1 & 1
\end{array}
\hspace{10mm} 
\begin{array}{c|ccccc}
\rs & 1 & a & b & c \\ \hline
1 & 1 & a & b & c \\ 
a & 1 & 1 & 1 & b \\ 
b & 1 & 1 & 1 & 1 \\ 
c & 1 & 1 & 1 & 1
\end{array}
. 
\]
The structure $(A,\ra,\rs,1)$ is a pseudo BE-algebra (\cite{Rez1}). Since $a\ra b=1$ and $b\ra a=1$, but $a\neq b$, 
axiom $(psBCK_6)$ is not satisfied, so $A$ is not a pseudo BCK-algebra. 
Consider the maps $\exists_i, \forall_i:A\longrightarrow A$, $i=1,2,3$, given in the table below:
\[
\begin{array}{c|cccccc}
 x & 1 & a & b & c  \\ \hline
\exists_1 x & 1 & a & b & c \\
\forall_1 x & 1 & a & b & c \\ \hline
\exists_2 x & 1 & b & b & b \\
\forall_2 x & 1 & b & b & b \\ \hline
\exists_3 x & 1 & 1 & c & c \\
\forall_3 x & 1 & c & c & c  
\end{array}
.   
\]
Then $\mathcal{MOP}(A)=\{(\exists_1,\forall_1),(\exists_2,\forall_2),(\exists_3,\forall_3)\}$.
\end{example}

\begin{example} \label{mpsBE-10-30}
Let $A=\{1,a,b,c,d\}$. Define the operations $\ra$ and $\rs$ on $A$ as follows:
\begin{eqnarray*}
\begin{array}{c|ccccc} \ra & 1 & a & b & c & d\\
\hline 1 & 1 & a & b & c & d\\
       a & 1 & 1 & c & c & 1 \\
       b & 1 & d & 1 & 1 & d\\
       c & 1 & d & 1 & 1 & d \\
       d & 1 & 1 & c & c & 1
\end{array}
\hspace{2cm}
\begin{array}{c|ccccc} \rs & 1 & a & b & c & d \\
\hline 1 & 1 & a & b & c  & d \\
       a & 1 & 1 & b & c & 1 \\
       b & 1 & d & 1 & 1 & d \\
       c & 1 & d & 1 & 1 & d \\
       d & 1 & 1 & b & c & 1      
\end{array}
.
\end{eqnarray*}
The structure $(A, \ra, \rs, 1)$ is a pseudo BE(A)-algebra.  
Moreover, $b\ra c=1$ and $c\ra b=1$, but $b\ne c$, so axiom $(psBCK_6)$ is not satisfied, hence $A$ is 
not a pseudo BCK-algebra.  
Consider the maps $\exists_i, \forall_i:A\longrightarrow A$, $i=1,2,3,4$, given in the table below:
\[
\begin{array}{c|cccccc}
 x & 1 & a & b & c  & d\\ \hline
\exists_1 x & 1 & a & b & c & d \\
\forall_1 x & 1 & a & b & c & d \\ \hline
\exists_2 x & 1 & a & c & c & d \\
\forall_2 x & 1 & a & c & c & d \\ \hline
\exists_3 x & 1 & d & b & c & d \\
\forall_3 x & 1 & d & b & c & d \\ \hline
\exists_4 x & 1 & d & c & c & d \\
\forall_4 x & 1 & d & c & c & d  
\end{array}
.   
\]
Then $\mathcal{MOP}(A)=\{(\exists_1,\forall_1),(\exists_2,\forall_2),(\exists_3,\forall_3),(\exists_4,\forall_4)\}$.
\end{example}

\begin{proposition} \label{mpsBE-20} Let $(A,\exists,\forall)$ be a monadic pseudo BE-algebra. Then the following hold, for all $x, y\in A:$ \\
$(1)$ $\exists 1=1;$ \\
$(2)$ $\forall 1=1;$ \\
$(3)$ $\forall\exists x=\exists x;$ \\
$(4)$ $\forall x=x$ iff $\exists x=x;$ \\
$(5)$ $\exists\exists x=\exists x;$ \\
$(6)$ $\forall\forall x=\forall x;$ \\
$(7)$ $\forall(\exists x\ra \exists y)=\exists x\ra \exists y$, 
      $\forall(\exists x\rs \exists y)=\exists x\rs \exists y;$ \\ 
$(8)$ $x\le \exists y$ iff $\exists x\le \exists y$ and $\forall x\le y$ iff $\forall x\le \forall y;$ \\
$(9)$ $\forall(\forall x\ra y)=\forall x\ra \forall y$, 
       $\forall(\forall x\rs y)=\forall x\rs \forall y;$ \\
$(10)$ $\forall(\forall x\ra \exists y)=\forall x\ra \exists y$, 
      $\forall(\forall x\rs \exists y)=\forall x\rs \exists y;$ \\
$(11)$ $\forall(x\ra \forall y)=\exists x\ra \forall y$, 
       $\forall(x\rs \forall y)=\exists x\rs \forall y;$ \\
$(12)$ $\forall(\forall x\ra \forall y)=\forall x\ra \forall y$, 
       $\forall(\forall x\rs \forall y)=\forall x\rs \forall y;$ \\
$(13)$ $\forall(\exists x\ra \exists y)=\exists x\ra \exists y$, 
       $\forall(\exists x\rs \exists y)=\exists x\rs \exists y;$ \\
$(14)$ $\exists(\exists x\ra \exists y)\le \exists x\ra \exists y$, 
       $\exists(\exists x\rs \exists y)\le \exists x\rs \exists y;$ \\
$(15)$ $\forall x=1$ iff $x=1$. 
\end{proposition}
\begin{proof}
$(1)$ It follows from $(M_1)$. \\
$(2)$ By $(M_1)$ and $(M_3)$ we have $\forall 1=\forall(x\ra \exists x)=\exists x\ra \exists x=1$. \\
$(3)$ By $(M_3)$ we have $\forall\exists x=\forall(1\ra \exists x)=\exists 1\ra \exists x=1\ra \exists x=\exists x$. \\
$(4)$ From $\forall x=x$, by $(M_5)$ we get $x=\forall x=\exists\forall x=\exists x$. 
Conversely, if $\exists x=x$, then using $(3)$ we have $x=\exists x=\forall\exists x=\forall x$. \\
$(5)$ From $\forall\exists x=\exists x$, by $(4)$ we get $\exists\exists x=\exists x$. \\
$(6)$ By $\exists\forall x=\forall x$, using $(4)$ we get $\forall\forall x=\forall x$. \\ 
$(7)$ Using $(M_3)$ and $(5)$ we have $\forall(\exists x\ra \exists y)=\exists\exists x\ra \exists y=
\exists x\ra \exists y$. \\
$(8)$ From $x\le \exists y$, using $(M_3)$ it follows that $\exists x\ra \exists y=\forall (x\ra \exists y)=
\forall 1=1$, hence $\exists x\le \exists y$. 
Conversely, if $\exists x\le \exists y$, then by $(M_2)$ we get 
$1=\exists x\ra \exists y=\forall(x\ra \exists y)\le x\ra \exists y$, hence $x\ra \exists y=1$, 
that is $x\le \exists y$. \\
By $\forall x\le y$, using $(M_5)$ and $(M_4)$ we have 
$\forall x\ra \forall y=\exists\forall x\ra \forall y=\forall(\exists\forall x\ra y)=
\forall(\forall x\ra y)=\forall 1=1$, so $\forall x\le \forall y$. 
Conversely, if $\forall x\le \forall y$, using $(M_5)$ and $(M_2)$ we get 
$1=\forall x\ra \forall y=\exists\forall x\ra \forall y=\forall(\exists\forall x\ra y)=\forall(\forall x\ra y)\le 
\forall x\ra y$, so $\forall x \ra y=1$. Hence $\forall x\le y$. \\
$(9)$ Using $(M_5)$ and $(M_4)$ we get 
$\forall(\forall x\ra y)=\forall(\exists\forall x\ra y)=\exists\forall x\ra \forall y=\forall x\ra \forall y$. 
Similarly, $\forall(\forall x\rs y)=\forall x\rs \forall y$. \\
$(10)$ By $(M_3)$ and $(M_5)$ we have 
$\forall(\forall x\ra \exists y)=\exists\forall x\ra \exists y=\forall x\ra \exists y$. 
Similarly, $\forall(\forall x\rs \exists y)=\forall x\rs \exists y$. \\
$(11)$ By $(M_3)$ and $(M_5)$ it follows that 
$\forall(x\ra \forall y)=\forall(x\ra \exists\forall y)=\exists x\ra \exists\forall y=\exists x\ra \forall y$. 
Similarly, $\forall(x\rs \forall y)=\exists x\rs \forall y$. \\
$(12)$ It follows by $(11)$ replacing $x$ by $\forall x$ and using $(M_5)$: 
$\forall(\forall x\ra \forall y)=\exists\forall x\ra \forall y=\forall x\ra \forall y$ and 
$\forall(\forall x\rs \forall y)=\exists\forall x\rs \forall y=\forall x\rs \forall y$. \\
$(13)$ It follows by $(10)$ replacing $x$ by $\forall x$ and using $(3)$. \\
$(14)$ Using $(M_3)$, $(M_4)$, $(M_5)$ and $(M_2)$ we have: \\ 
$\hspace*{1.6cm}$ $\exists(\exists x\ra \exists y)\rs (\exists x\ra \exists y)=
                   \exists\forall(x\ra \exists y)\rs \forall(x\ra \exists y)$ \\
$\hspace*{6cm}$ $=\forall(\exists\forall(x\ra \exists y)\rs (x\ra \exists y))$ \\
$\hspace*{6cm}$ $=\forall(\forall(x\ra \exists y)\rs (x\ra \exists y))=\forall 1=1$, \\
hence $\exists(\exists x\ra \exists y)\le \exists x\ra \exists y$. 
Similarly, $\exists(\exists x\rs \exists y)\le \exists x\rs \exists y$. \\
$(15)$ If $x=1$, then by $(2)$, $\forall x=1$. Conversely, $\forall x=1$ implies $1=\forall x\le x$, so $x=1$.
\end{proof}

\begin{proposition} \label{mpsBE-25} Let $(A,\exists,\forall)$ be a monadic pseudo BE(T)-algebra. 
Then, the quantifiers $\forall$ and $\exists$ are isotone, that is for all $x, y\in A$ such that $x\le y$, 
we have $\forall x\le \forall y$ and $\exists x\le \exists y$. 
\end{proposition}
\begin{proof}
By $(M_2)$ we get $\forall x\le x\le y$, so by transitivity, $\forall x\ra y=1$. 
Applying $(M_5)$ and $(M_4)$ we have $\forall x\ra \forall y=\exists\forall x\ra \forall y=
\forall(\exists\forall x\ra y)=\forall(\forall x\ra y)=\forall 1=1$, hence $\forall x\le \forall y$. 
Similarly, using $(M_1)$ and $(M_3)$ we have $x\le y\le \exists y$, and we get 
$\exists x\ra \exists y=\forall (x\ra \exists y)=\forall 1=1$, thus $\exists x\le \exists y$. 
\end{proof}

\begin{corollary} \label{mpsBE-25-10} The following hold: \\
$(1)$ if $(A,\exists,\forall)$ is a monadic pseudo BE(A)-algebra, then $\forall$ and $\exists$ are isotone; \\
$(2)$ if $(A,\exists,\forall)$ is a monadic pseudo BCK-algebra, then $\forall$ and $\exists$ are isotone; \\
$(3)$ if $(A,\exists,\forall)$ is a monadic commutative pseudo BE-algebra, then $\forall$ and $\exists$ are isotone.   
\end{corollary}

\begin{remark} \label{mpsBE-20-10}
Let $(A,\exists,\forall)$ be a monadic pseudo BCK-algebra such that $\forall x=\exists x$, for all $x\in A$. 
From $\forall x\le x\le \exists x$, for all $x\in A$, it follows that $\forall=\exists=Id_A$. 
As we can see in Examples \ref{mpsBE-10-20} and \ref{mpsBE-10-30}, this property is no longer valid for the case of 
proper monadic pseudo BE-algebras.
\end{remark}

\begin{remark} \label{mpsBE-20-20} 
A closure operator on a poset $(A,\le)$ is a map $\gamma$ that is increasing, monotone and idempotent, i.e. 
$x\le \gamma(x)$, $x\le y$ implies $\gamma(x)\le \gamma(y)$ and $\gamma(\gamma(x))=\gamma(x)$, for all $x\in A$. 
Dually, an interior operator is a decreasing ($x\le \gamma(x)$), monotone and idempotent map on $(A,\le)$ 
(see \cite{Gal3}). 
Clearly, if $(A,\exists,\forall)$ is a monadic pseudo BE-algebra, then $\exists$ is a closure operator, and $\forall$ is an interior operator on $(A,\le)$. 
\end{remark}

\begin{remark} \label{mpsBE-40-10}
Given two posets $A$, $B$, we say that the isotone maps $f:A\longrightarrow B$ and $g:B\longrightarrow A$ form a 
\emph{residuated pair} or an \emph{adjunction} or a \emph{Galois connection} if $f(x)\le y$ iff $x\le g(y)$, 
for all $x\in A$, $y\in B$. For a map $f:A\longrightarrow B$, if there exists a map 
$g:B\longrightarrow A$ such that the pair $(f,g)$ is residuated, then $g$ is unique (see \cite{Gal3}). \\
If $(A,\exists,\forall)$ is a monadic pseudo BE(T)-algebra, then $(\exists,\forall)$ is a residuated pair over $A$. 
Indeed, if $\exists x\le y$, then $\forall\exists x\le \forall y$, that is $\exists x\le \forall y$. 
Since $x\le \exists x$ and $A$ has the (T) property, we get $x\le \forall y$. 
Conversely, suppose that $x\le \forall y$. It follows that $\exists x\le \exists\forall y=\forall y\le y$, 
hence $\exists x\le y$. 
Moreover, by Proposition \ref{mpsBE-25}, the maps $\exists$ and $\forall$ are isotone, so that 
$(\exists,\forall)$ is a residuated pair. 
\end{remark}

\begin{proposition} \label{mpsBE-50} Let $(A,\exists,\forall)$ be a monadic bounded pseudo BE-algebra. Then the following hold, for all $x\in A:$ \\
$(1)$ $\forall 0=0;$ \\
$(2)$ $\exists 0=0;$ \\
$(3)$ $(\exists x)^{-}=\forall x^{-}$, $(\exists x)^{\sim}=\forall x^{\sim};$ \\
$(4)$ $\forall(\exists x)^{-}=(\exists x)^{-}$, $\forall(\exists x)^{\sim}=(\exists x)^{\sim}$. \\
$(5)$ $\forall(x^{-\sim})=(\exists x^{-})^{\sim}$, $\forall(x^{\sim-})=(\exists x^{\sim})^{-};$ \\ 
$(6)$ $\forall((\forall x^{-})^{\sim})=(\forall x^{-})^{\sim}$, 
      $\forall((\forall x^{\sim})^{-})=(\forall x^{\sim})^{-};$ \\
$(7)$ $\forall(\forall x)^{-}=(\forall x)^{-}$, $\forall(\forall x)^{\sim}=(\forall x)^{\sim};$ \\
$(8)$ $\exists(\exists x)^{-}=(\exists x)^{-}$, $\exists(\exists x)^{\sim}=(\exists x)^{\sim};$ \\
$(9)$ $\exists x=0$ iff $x=0$. 
\end{proposition}
\begin{proof}
$(1)$ It follows from $(M_2)$ for $x:=0$. \\
$(2)$ It is a consequence of Proposition \ref{mpsBE-20}$(4)$. \\
$(3)$ Using $(M_3)$ we get  
$(\exists x)^{-}=\exists x\ra 0=\exists x\ra \exists 0=\forall (x\ra \exists 0)=\forall (x\ra 0)=\forall x^{-}$. 
Similarly, $(\exists x)^{\sim}=\forall x^{\sim}$. \\
$(4)$ By $(M_4)$ we have 
$\forall(\exists x)^{-}=\forall(\exists x\ra 0)=\exists x\ra \forall 0=\exists x\ra 0=(\exists x)^{-}$. 
Similarly, $\forall(\exists x)^{\sim}=(\exists x)^{\sim}$. \\
$(5)$ Applying $(M_3)$, it follows that 
$\forall(x^{-\sim})=\forall(x^{-}\rs 0)=\forall(x^{-}\rs \exists 0)=\exists x^{-}\rs \exists 0=
                    \exists x^{-}\rs 0=(\exists x^{-})^{\sim}$. 
Similarly, $\forall(x^{\sim-})=(\exists x^{\sim})^{-}$. \\
$(6)$ Using $(M_3)$ and $(M_5)$ we get 
$\forall((\forall x^{-})^{\sim})=\forall(\forall x^{-}\rs 0)=\forall(\forall x^{-}\rs \exists 0)=
          \exists\forall x^{-}\rs \exists 0=\forall x^{-}\rs \exists 0=\forall x^{-}\rs 0=
          (\forall x^{-})^{\sim}$. 
Similarly, $\forall((\forall x^{\sim})^{-})=(\forall x^{\sim})^{-}$. \\
$(7)$ Using $(M_3)$ and $(M_5)$ we get 
$\forall(\forall x)^{-}=\forall(\forall x\ra 0)=\forall(\forall x\ra \exists 0)= \exists\forall x\ra \exists 0=
\forall x\ra 0=(\forall x)^{-}$.    
Similarly, $\forall(\forall x)^{\sim}=(\forall x)^{\sim}$. \\
$(8)$ Applying $(M_3)$, $(M_5)$ and $(3)$ we have 
$\exists(\exists x)^{-}=\exists(\exists x\ra 0)=\exists(\exists x\ra \exists 0)=\exists\forall(x\ra \exists 0)=
\forall(x\ra \exists 0)=\forall x^{-}=(\exists x)^{-}$. 
Similarly, $\exists(\exists x)^{\sim}=(\exists x)^{\sim}$. \\
$(9)$ By $(2)$, $x=0$ implies $\exists x=0$. Conversely, $\exists x=0$ implies $x\le \exists x=0$, so $x=0$. 
\end{proof}

If $(A,\exists,\forall)$ is a monadic pseudo BE-algebra, denote  
$A_{\exists\forall}=\{x\in A \mid x=\exists x\}=\{x\in A \mid x=\forall x\}$. 
We say that $A_{\exists\forall}$ is the set of \emph{fixed elements} of $(A,\exists,\forall)$.  

\begin{proposition} \label{mpsBE-80} Let $(A,\exists,\forall)$ be a monadic pseudo BE-algebra. 
Then the following hold: \\
$(1)$ $A_{\exists\forall}=\forall A_{\exists\forall}=\exists A_{\exists\forall}=\forall A=\exists A;$ \\
$(2)$ $A_{\exists\forall}=A$ iff $\forall=\exists=Id_A;$ \\
$(3)$ $A_{\exists\forall}$ is a subalgebra of $A;$ \\
$(4)$ $\forall A$ and $\exists A$ are subalgebras of $A;$ \\
$(5)$ if $A$ has the (T) property and $\forall A$ is bounded, then $A$ is bounded. 
\end{proposition}
\begin{proof}
$(1)$ If $x\in A_{\exists\forall}$, then $x=\forall x$, so $x\in \forall A_{\exists\forall}$. 
Hence, $A_{\exists\forall}\subseteq \forall A_{\exists\forall}$. 
Conversely, let $x\in \forall A_{\exists\forall}$, that is $x=\forall y$, with $y\in A_{\exists\forall}$. 
Then, we have $x=\forall y=\forall\forall y=\forall x$, hence $x\in A_{\exists\forall}$. 
Thus, $\forall A_{\exists\forall}\subseteq A_{\exists\forall}$. 
It follows that $A_{\exists\forall}=\forall A_{\exists\forall}$, and similarly 
$A_{\exists\forall}=\exists A_{\exists\forall}$. 
If $x\in A_{\exists\forall}$, then $x=\forall x$, so $x\in \forall A$, hence $A_{\exists\forall}\subseteq \forall A$. 
Conversely, let $x\in \forall A$, that is $x=\forall y$, with $y\in A$. 
It follows that $x=\forall y=\exists\forall y=\exists x$, so $x\in A_{\exists\forall}$, thus 
$\forall A\subseteq A_{\exists\forall}\subseteq$, so $A_{\exists\forall}=\forall A$. 
Similarly, $A_{\exists\forall}=\exists A$. \\
$(2)$ We have $A_{\exists\forall}=A$ iff $\forall x=x$ and $\exists x=x$, for all $x\in A$. 
Hence $\forall x=\exists x=Id_A(x)$, for all $x\in A$, that is $\forall=\exists=Id_A$. \\
$(3)$ Since $1=\forall 1$, we have $1\in A_{\exists\forall}$. 
Let $x, y \in A_{\exists\forall}$, that is $x=\exists x$ and $y=\exists y$. 
Using $(M_3)$ we get $\forall(x\ra y)=\forall(\exists x\ra \exists y)=\exists\exists x\ra \exists y=
\exists x\ra \exists y=x\ra y$, hence $x\ra y\in A_{\exists\forall}$. 
Similarly, $x\rs y\in A_{\exists\forall}$, so $A_{\exists\forall}$ is a subalgebra of $A$. \\
$(4)$ It follows by $(1)$ and $(3);$ \\
$(5)$ Let $0$ be the smallest element of $\forall A$, so we have $0\le \forall x\le x$, for all $x\in A$. 
Since $A$ has the (T) property, it follows that $0\le x$, for all $x\in A$, that is $0$ is the smallest 
element of $A$.
\end{proof}

\begin{proposition} \label{mpsBE-90} Let $(A,\exists,\forall)$ be a monadic pseudo BE-algebra. 
Then the following hold: \\
$(1)$ $\Img(\forall)=A_{\exists\forall};$ \\
$(2)$ if $\forall$ is surjective, then $\forall=Id_A;$ \\
$(3)$ $\Ker(\forall)=\{1\}$, where $\Ker(\forall)=\{x\in A \mid \forall x=1\}$. 
\end{proposition}
\begin{proof}
$(1)$ Obviously, $A_{\exists\forall}\subseteq \Img(\forall)$.  
If $y\in \Img(\forall)$, then there exists $x\in A$ such that $y=\forall x$. It follows that 
$\forall y=\forall\forall x=\forall x=y$, so $\Img(\forall)\subseteq A_{\exists\forall}$. 
Hence $\Img(\forall)=A_{\exists\forall}$. \\
$(2)$ Let $x\in A$. Since $A=\Img(\forall)$, then there exists $x^{\prime}\in A$ such that $x=\forall x^{\prime}$. 
It follows that $\forall x=\forall\forall x^{\prime}=\forall x^{\prime}=x$. Hence $\forall = Id_A$. \\
$(3)$ It is a consequence of Proposition \ref{mpsBE-20}$(15)$. 
\end{proof}

\begin{example} \label{mpsBE-100}
Consider the pseudo BE-algebra and its monadic operators from Example \ref{mpsBE-10-30}. 
We have $A_{\exists_1\forall_1}=A=\Img(\forall_1)$, $A_{\exists_2\forall_2}=\{a,c,d,1\}=\Img(\forall_2)$, $A_{\exists_3\forall_3}=\{b,c,d,1\}=\Img(\forall_3)$, $A_{\exists_4\forall_4}=\{c,d,1\}=\Img(\forall_4)$.
\end{example}


$\vspace*{5mm}$

\section{Monadic bounded commutative pseudo BE-algebras}

In this section we define the monadic bounded commutative pseudo BE-algebras and we investigate their properties. 
We give a method to construct a monadic operator $(\exists,\forall)$ on these structures. 
It is proved that, in monadic bounded commutative pseudo BE-algebras, there is a one-to-one correspondence 
between universal and existential quantifiers. 

\begin{definition} \label{mcomm-psBE-10} $\rm($\cite{Ciu33}$\rm)$ A pseudo BE-algebra $(A,\ra,\rs,1)$ is said to be \emph{commutative} if it satisfies the following conditions, for all $x, y\in A$: \\ 
$\hspace*{3cm}$ $(x\ra y)\rs y=(y\ra x)\rs x,$ \\
$\hspace*{3cm}$ $(x\rs y)\ra y=(y\rs x)\ra x$. 
\end{definition} 

If we denote $x\Cup_1 y=(x\ra y)\rs y$ and $x\Cup_2 y=(x\rs y)\ra y$, then $A$ is commutative if 
$x\Cup_1 y=y\Cup_1 x$ and $x\Cup_2 y=y\Cup_2 x$, for all $x, y\in A$. 

\begin{example} \label{comm-psBE-20} The pseudo BE-algebra $(A,\ra,\rs,1)$ from Example \ref{mpsBE-10-30} 
is not commutative, since $(a\rs c)\ra c=b\neq c=(c\rs a)\ra a$.
\end{example}

\begin{example} \label{mcomm-psBE-30} $\rm($\cite{Ciu33}$\rm)$
Let $(G, \vee,\wedge, \cdot, ^{-1}, e)$ be an $\ell$-group. On the negative cone $G^{-}=\{g\in G \mid g\le e\}$ we define the operations $x\rightarrow y=y\cdot (x\vee y)^{-1}$, $x\rightsquigarrow y=(x\vee y)^{-1}\cdot y$. 
Then $(G^{-}, \rightarrow, \rightsquigarrow, e)$ is a commutative pseudo BE-algebra. 
\end{example}

\begin{theorem} \label{mcomm-psBE-40} $\rm($\cite{Ciu33}$\rm)$ The class of commutative pseudo BE-algebras is 
equivalent to the class of commutative pseudo BCK-algebras. 
\end{theorem}

As a consequence of this theorem, all results holding for commutative pseudo BCK-algebras also hold for commutative pseudo BE-algebras. 
According to \cite[Corollary 4.1.6]{Kuhr6}, there are no proper finite commutative pseudo BCK-algebras. 
It folows that there are no proper finite commutative pseudo BE-algebras, in the other words, 
every finite commutative pseudo BE-algebra is a commutative BE-algebra. 

\begin{remark} \label{mcomm-psBE-40-10}
Obviously, by $(psBE_1)$ the relation ``$\le$" is reflexive. If $A$ is commutative, then from $x\le y$ and $y\le x$ 
we have $x\ra y=1$ and $y\ra x=1$, so we get $x=1\rs x=(y\ra x)\rs x=(x\ra y)\rs y=1\rs y=y$, hence ``$\le$" is antisymmetric. 
Moreover, if $A$ is distributive, then from 
$x\le y$ and $y\le z$ we have $x\ra z=1\rs (x\ra z)=(x\ra y)\rs (x\ra z)=x\ra (y\rs z)=x\rs 1=1$, so $x\ra z=1$, 
hence the relation ``$\le$" is transitive. \\ 
On the other hand, any commutative pseudo BE-algbra is a pseudo BCK-algebra, so that the relation ``$\le$" 
is a partial order on $A$.
\end{remark}

\begin{definition} \label{minv-psBE-10} Let $(A, \ra, \rs, 0, 1)$ be a bounded pseudo BE-algebra. Then: \\
$(1)$ $A$ is called \emph{good} if $x^{-\sim}=x^{\sim-}$, for all $x\in A$. \\ 
$(2)$ $A$ is called \emph{involutive} if $x^{-\sim}=x^{\sim-}=x$, for all $x\in A$.  
\end{definition}

We recall some results from \cite{Ior1}: \\
$(1)$ Any bounded commutative pseudo BE-algebra is involutive, so by Theorem \ref{mcomm-psBE-40} is an involutive 
pseudo BCK-algebra, and it satisfies the (pP) condition. \\ 
$(2)$ If $A$ is a bounded commutative pseudo BE-algebra, then $(A,\le)$ is a lattice with $\vee$ and $\wedge$ 
defined by $x\vee y=(x\ra y)\rs y=(x\rs y)\ra y$, $x\wedge y=(x^{-}\vee y^{-})^{\sim}=(x^{\sim}\vee y^{\sim})^{-}$, 
for all $x, y\in A$. \\
$(3)$ The bounded commutative pseudo BE-algebras coincide (are categorically isomorphic) with pseudo MV-algebras. 



In a bounded commutative pseudo BE-algebra we define $x\oplus y=y^{\sim}\ra x=x^{-}\rs y$, for all $x, y\in A$. 

\begin{lemma} \label{minv-psBE-30-10} $\rm($\cite{Ior16}$\rm)$
In any bounded commutative pseudo BE-algebra the following hold for all $x, y\in A:$ \\
$(1)$ $x\oplus y=(y^{-}\odot x^{-})^{\sim}=(y^{\sim}\odot x^{\sim})^{-};$ \\
$(2)$ $x\odot y=(y^{-}\oplus x^{-})^{\sim}=(y^{\sim}\oplus x^{\sim})^{-}$.
\end{lemma}

\begin{definition} \label{minv-psBE-40} Let $(A,\ra,\rs,0,1,\exists,\forall)$ be an algebra of type $(2,2,0,0,1,1)$ such that 
$(A,\ra,\rs,0,1)$ is a bounded commutative pseudo BE-algebra. Then $(A,\exists,\forall)$ is called a 
\emph{monadic bounded commutative pseudo BE-algebra} if it satisfes conditions $(M_1)$-$(M_5)$ and the 
following conditions, for all $x\in A:$ \\
$(M_6)$ $\forall(x\odot x)=\forall x\odot \forall x;$ \\
$(M_7)$ $\forall(x\oplus x)=\forall x\oplus \forall x$.   
\end{definition}

\begin{example} \label{minv-psBE-40-10} 
Consider the set $A=\{1,a,b,0\}$ and the operation $\ra$ given by the following table:
\[
\hspace{10mm}
\begin{array}{c|ccccc}
\ra & 1 & a & b & 0 \\ \hline
1 & 1 & a & b & 0 \\ 
a & 1 & 1 & b &b \\ 
b & 1 & a & 1 & a \\ 
0 & 1 & 1 & 1 & 1
\end{array}
. 
\]
The structure $(A,\ra,0,1)$ is a bounded commutative BE-algebra, so it is a bounded commutative BCK-algebra. 
The operations $\odot$ and $\oplus$ are given in the following tables:
\[
\hspace{10mm}
\begin{array}{c|ccccc}
\odot & 1 & a & b & 0 \\ \hline
1 & 1 & a & b & 0 \\ 
a & a & a & 0 & 0 \\ 
b & b & 0 & b & 0 \\ 
0 & 0 & 0 & 0 & 0
\end{array}
\hspace{10mm} 
\begin{array}{c|ccccc}
\oplus & 1 & a & b & 0 \\ \hline
1 & 1 & 1 & 1 & 1 \\ 
a & 1 & a & 1 & a \\ 
b & 1 & 1 & b & b \\ 
0 & 1 & a & b & 0
\end{array}
. 
\]
Consider the maps $\exists_i, \forall_i:A\longrightarrow A$, $i=1,2$, given in the table below:
\[
\begin{array}{c|cccccc}
 x & 1 & a & b & 0  \\ \hline
\exists_1 x & 1 & a & b & 0 \\
\forall_1 x & 1 & a & b & 0 \\ \hline
\exists_2 x & 1 & 1 & 1 & 0 \\
\forall_2 x & 1 & 0 & 0 & 0 \\ 
\end{array}
.   
\]
Then $\mathcal{MOP}(A)=\{(\exists_1,\forall_1),(\exists_2,\forall_2)\}$.
\end{example}

For the rest of this section $(A,\ra,\rs,0,1,\exists,\forall)$ will be a monadic bounded commutative pseudo BE-algebra, 
unless otherwise stated. 

\begin{proposition} \label{minv-psBE-50} The following hold, for all $x\in A:$ \\
$(1)$ $\exists x=(\forall x^{-})^{\sim}=(\forall x^{\sim})^{-}$, 
      $\forall x=(\exists x^{-})^{\sim}=(\exists x^{\sim})^{-};$ \\
$(2)$ $\exists((\exists x^{-})^{\sim})=(\exists x^{-})^{\sim}$,  
      $\exists((\exists x^{\sim})^{-})=(\exists x^{\sim})^{-};$ \\
$(3)$ $\exists x^{-}=(\forall x)^{-}$, $\exists x^{\sim}=(\forall x)^{\sim}$. 
\end{proposition}
\begin{proof}
$(1)$ From Proposition \ref{mpsBE-50}$(3)$ we have $(\exists x)^{-}=\forall x^{-}$ and 
$(\exists x)^{\sim}=\forall x^{\sim}$, so $\exists x=(\forall x^{-})^{\sim}=(\forall x^{\sim})^{-}$. 
Replacing $x$ by $x^{\sim}$ in $\exists x=(\forall x^{-})^{\sim}$ we get $\exists x^{\sim}=(\forall x)^{\sim}$.  
Similarly, replacing $x$ by $x^{-}$ in $\exists x=(\forall x^{\sim})^{-}$ we get $\exists x^{-}=(\forall x)^{-}$.  
Hence, $\forall x=(\exists x^{-})^{\sim}=(\exists x^{\sim})^{-}$. \\
$(2)$ Using $(1)$ and $(M_5)$ we have 
$\exists((\exists x^{-})^{\sim})=\exists\forall x=\forall x=(\exists x^{-})^{\sim}$ and 
$\exists((\exists x^{\sim})^{-})=\exists\forall x=\forall x=(\exists x^{\sim})^{-}$. \\
$(3)$ Replacing $x$ by $x^{-}$ in the identity $\exists x=(\forall x^{\sim})^{-}$ from $(1)$, we get 
$\exists x^{-}=(\forall x)^{-}$. 
Similarly, replacing $x$ by $x^{\sim}$ in the identity $\exists x=(\forall x^{-})^{\sim}$ we have 
$\exists x^{\sim}=(\forall x)^{\sim}$.
\end{proof}


\begin{proposition} \label{minv-psBE-70} The following hold, for all $x, y\in A:$ \\ 
$(1)$ $\forall (x\wedge y)=\forall x\wedge \forall y$ iff $\exists (x\vee y)=\exists x\vee \exists y;$ \\
$(2)$ $\forall (x\odot y)=\forall x\odot \forall y$ iff $\exists (x\oplus y)=\exists x\oplus \exists y;$ \\
$(3)$ $\forall (x\oplus y)=\forall x\oplus \forall y$ iff $\exists (x\odot y)=\exists x\odot \exists y$. 
\end{proposition}
\begin{proof}
The proof uses Proposition \ref{mpsBE-50}$(3)$ and the relationships between the operations defined on an bounded commutative pseudo BE-algebra. \\ 
$(1)$ If $\forall (x\wedge y)=\forall x\wedge \forall y$, we have: \\
$\hspace*{2.0cm}$ $\exists(x\vee y)=\exists(x^{-}\wedge y^{-})^{\sim}=(\forall(x^{-}\wedge y^{-}))^{\sim}=
(\forall x^{-}\wedge y^{-})^{\sim}$ \\
$\hspace*{3.5cm}$ $=((\exists x)^{-}\wedge (\exists y)^{-}))^{\sim}=\exists x\vee \exists y$. \\
Similarly, $\exists (x\vee y)=\exists x\vee \exists y$ implies $\forall (x\wedge y)=\forall x\wedge \forall y$. \\
$(2)$ Suppose that $\forall (x\odot y)=\forall x\odot \forall y$, we get: \\ 
$\hspace*{2.0cm}$ $\exists(x\oplus y)=\exists(x^{-}\odot y^{-})^{\sim}=(\forall(x^{-}\odot y^{-}))^{\sim}=
(\forall x^{-}\odot \forall y^{-})^{\sim}$ \\
$\hspace*{3.5cm}$ $=((\exists x)^{-}\odot (\exists y)^{-}))^{\sim}=\exists x\oplus \exists y$. \\
Similarly, $\exists (x\oplus y)=\exists x\oplus \exists y$ implies $\forall (x\odot y)=\forall x\odot \forall y$. \\
$(3)$ If $\forall (x\oplus y)=\forall x\oplus \forall y$, we have: \\
$\hspace*{2.0cm}$ $\exists(x\odot y)=\exists(x^{-}\oplus y^{-})^{\sim}=(\forall(x^{-}\oplus y^{-}))^{\sim}=
(\forall x^{-}\oplus \forall y^{-})^{\sim}$ \\
$\hspace*{3.5cm}$ $=((\exists x)^{-}\oplus (\exists y)^{-}))^{\sim}=\exists x\odot \exists y$. \\
Similarly, $\exists (x\odot y)=\exists x\odot \exists y$ implies $\forall (x\oplus y)=\forall x\oplus \forall y$. 
\end{proof}

Inspired from \cite{Ior16}, in the following two theorems we construct monadic operators $(\exists,\forall)$ 
on a bounded commutative pseudo BE-algebra.   

\begin{theorem} \label{minv-psBE-200} 
Let $(A,\ra,\rs,0,1)$ be a bounded commutative pseudo BE-algebra and let $\tau:A\longrightarrow A$ satisfying 
the following conditions, for all $x, y\in A:$ \\
$(U_1)$ $\tau x\le x;$ \\
$(U_2)$ $(\tau x^{-})^{\sim}=(\tau x^{\sim})^{-};$ \\
$(U_3)$ $\tau (x\oplus (\tau y)^{-})=\tau x\oplus (\tau y)^{-}$,  
        $\tau ((\tau x)^{\sim}\oplus y)= (\tau x)^{\sim}\oplus \tau y;$ \\
$(U_4)$ $\tau (x\oplus \tau y)=\tau (\tau x\oplus y)=\tau x\oplus \tau y;$ \\
$(U_5)$ $\tau((x^{-}\oplus x^{-})^{\sim})=((\tau x)^{-}\oplus (\tau x)^{-})^{\sim}$, 
        $\tau((x^{\sim}\oplus x^{\sim})^{-})=((\tau x)^{\sim}\oplus (\tau x)^{\sim})^{-};$ \\
$(U_6)$ $\tau(x\oplus x)=\tau x\oplus \tau x$. \\           
Consider the maps $\forall, \exists : A\longrightarrow A$ defined by $\forall x=\tau x$, 
$\exists x=(\tau x^{-})^{\sim}=(\tau x^{\sim})^{-}$, for all $x\in A$. Then $(A,\exists,\forall)$ is a monadic 
bounded commutative pseudo BE-algebra.
\end{theorem}
\begin{proof}
Obviously, axiom $(M_2)$ folows by $(U_1)$. 
By $(U_1)$ we also have $\tau x^{-}\le x^{-}$, so $x^{-\sim}\le (\tau x^{-})^{\sim}$, that is $x\le \exists x$, 
for all $x\in A$. Hence $(M_1)$ is satisfied. \\ 
It is easy to see that, from $\forall x\le x$ and $x\le \exists x$ we get $\forall 0=0$ and $\exists 1=1$. 
Replacing successively $x$ by $x^{\sim}$ and $x$ by $x^{-}$ in 
$\exists x=(\forall x^{-})^{\sim}=(\forall x^{\sim})^{-}$ we have 
$\forall x=(\exists x^{\sim})^{-}=(\exists x^{-})^{\sim}$. 
We also have $\exists x^{-}=(\forall x^{-\sim})^{-}=(\forall x)^{-}$ and similarly, 
$\exists x^{\sim}=(\forall x)^{\sim}$, $\forall x^{-}=(\exists x)^{-}$, $\forall x^{\sim}=(\exists x)^{\sim}$. \\ 
From the second equality of $(U_3)$ we have $\tau ((\tau y)^{\sim}\oplus x)= (\tau y)^{\sim}\oplus \tau x$, 
hence $\tau (x^{\sim}\ra (\tau y)^{\sim})=(\tau x)^{\sim}\ra (\tau y)^{\sim}$. Replacing $x$ by $x^{-}$ and 
$y$ by $y^{-}$ we get $\tau (x\ra (\tau y^{-})^{\sim})=(\tau x^{-})^{\sim}\ra (\tau y^{-})^{\sim}$.
It follows that $\forall (x\ra \exists y)=\exists x\ra \exists y$. 
Similarly, from $\tau (x\oplus (\tau y)^{-})=\tau x\oplus (\tau y)^{-}$ we get 
$\tau (x^{-}\rs (\tau y)^{-})=(\tau x)^{-}\rs (\tau y)^{-}$. 
Replacing $x$ by $x^{\sim}$ and $y$ by $y^{\sim}$, it follows that 
$\tau (x\rs (\tau y^{\sim})^{-})=(\tau x^{\sim})^{-}\rs (\tau y^{\sim})^{-}$, that is 
$\forall (x\rs \exists y)=\exists x\rs \exists y$, for all $x, y\in A$. Thus, axiom $(M_3)$ is verified. 
By $(U_4)$ we have $\tau(y\oplus \tau x)=\tau y\oplus \tau x$, so that 
$\tau ((\tau x)^{\sim}\ra y)=(\tau x)^{\sim}\ra \tau y$. Replacing $x$ by $x^{-}$ we get 
$\tau ((\tau x^{-})^{\sim}\ra y)=(\tau x^{-})^{\sim}\ra \tau y$, that is 
$\forall (\exists x\ra y)=\exists x\ra \forall y$. 
Similary, from $\tau (\tau x\oplus y)=\tau x\oplus \tau y$ we get 
$\tau((\tau x)^{-}\rs y)=(\tau x)^{-}\rs \tau y$, and replacing $x$ by $x^{\sim}$ we have 
$\tau((\tau x^{\sim})^{-}\rs y)=(\tau x^{\sim})^{-}\rs \tau y$, thus 
$\forall(\exists x\rs y)=\exists x\rs \forall y$, for all $x, y\in A$.  Hence, axiom $(M_4)$ is satisfied. 
Using $(M_3)$ and the identity $\exists x^{-}=(\forall x)^{-}$, it follows that \\
$\hspace*{2cm}$   $\exists \forall x=\exists ((\exists x^{\sim})^{-})=(\forall (\exists x^{\sim}))^{-}=
                   (\forall(1\ra \exists x^{\sim}))^{-}$ \\
$\hspace*{2.75cm}$ $=(\exists 1\ra \exists x^{\sim})^{-}=(1\ra \exists x^{\sim}))^{-}=
(\exists x^{\sim})^{-}=\forall x$, \\
and axiom $(M_5)$ is verified. \\
The identities from $(U_5)$ lead to $\forall(x\odot x)=\forall x\odot \forall x$, that is $(M_6)$. 
Finally, we can see that the identity from $(U_6)$ is in fact $(M_7)$. 
We conclude that $(A,\forall,\exists)$ is a monadic bounded commutative pseudo BE-algebra.
\end{proof}

\begin{theorem} \label{minv-psBE-210} 
Let $(A,\ra,\rs,0,1)$ be a bounded commutative pseudo BE-algebra and let $\sigma:A\longrightarrow A$ satisfying 
the following conditions, for all $x, y\in A:$ \\
$(E_1)$ $x\le \sigma x;$ \\
$(E_2)$ $(\sigma x^{-})^{\sim}=(\sigma x^{\sim})^{-};$ \\
$(E_3)$ $\sigma (x\odot (\sigma y)^{\sim})= \sigma x\odot (\sigma y)^{\sim}$, 
        $\sigma ((\sigma x)^{-}\odot y)=(\sigma x)^{-}\odot \sigma y;$ \\         
$(E_4)$ $\sigma (x\odot \sigma y)=\sigma (\sigma x\odot y)=\sigma x\odot \sigma y;$ \\
$(E_5)$ $\sigma((x^{-}\odot x^{-})^{\sim})=((\sigma x)^{-}\odot (\sigma x)^{-})^{\sim}$, 
        $\sigma((x^{\sim}\odot x^{\sim})^{-})=((\sigma x)^{\sim}\odot (\sigma x)^{\sim})^{-}$. \\
$(E_6)$ $\sigma(x\odot x)=\sigma x\odot \sigma x$. \\           
Consider the maps $\forall, \exists : A\longrightarrow A$ defined by $\exists x=\sigma x$,  
$\forall x=(\sigma x^{-})^{\sim}=(\sigma x^{\sim})^{-}$, for all $x\in A$. 
Then $(A,\exists,\forall)$ is a monadic bound commutative pseudo BE-algebra.
\end{theorem}
\begin{proof} 
Denote $\tau x=(\sigma x^{-})^{\sim}=(\sigma x^{\sim})^{-}$. 
We show that the axioms $(E_1)$-$(E_6)$ imply the axioms $(U_1)$-$(U_6)$.  
By $(E_1)$ we have $x^{-}\le \sigma x^{-}$, hence $(\sigma x^{-})^{\sim}\le x^{-\sim}=x$, that is $\tau x\le x$, 
so $(E_1)$ implies $(U_1)$. 
From $\tau x=(\sigma x^{-})^{\sim}$, replacing $x$ by $x^{\sim}$ we get $\tau x^{\sim}=(\sigma x)^{\sim}$, 
so $\sigma x=(\tau x^{\sim})^{-}$ and similary $\sigma x=(\tau x^{-})^{\sim}$. 
Hence $(\tau x^{-})^{\sim}=(\tau x^{\sim})^{-}$, and $(E_2)$ implies $(U_2)$. \\
It is easy to check that $\sigma x^{-}=(\tau x)^{-}$ and $\sigma x^{\sim}=(\tau x)^{\sim}$. 
Applying Lemma \ref{minv-psBE-30-10} and the definition of $\sigma$, the first equality of $(E_3)$ can be written 
successively: \\
$\hspace*{3cm}$ $\sigma (y\odot (\sigma x)^{\sim})= \sigma y\odot (\sigma x)^{\sim}$, \\
$\hspace*{3cm}$ $\sigma((\sigma x)^{\sim-}\oplus y^{-})^{\sim}=((\sigma x)^{\sim-}\oplus (\sigma y)^{-})^{\sim}$, \\
$\hspace*{3cm}$ $\sigma(\sigma x\oplus y^{-})^{\sim}=(\sigma x\oplus (\sigma y)^{-})^{\sim}$. \\
Replacing $x$ by $x^{\sim}$ and $y$ by $y^{\sim}$ it follows that: \\
$\hspace*{3cm}$ $\sigma(\sigma x^{\sim}\oplus y)^{\sim}=(\sigma x^{\sim}\oplus (\sigma y^{\sim})^{-})^{\sim}$, \\
$\hspace*{3cm}$ $\sigma((\tau x)^{\sim}\oplus y)^{\sim}=((\tau x)^{\sim}\oplus \tau y)^{\sim}$, \\
$\hspace*{3cm}$ $(\tau((\tau x)^{\sim}\oplus y))^{\sim}=((\tau x)^{\sim}\oplus \tau y)^{\sim}$, \\
$\hspace*{3cm}$ $\tau((\tau x)^{\sim}\oplus y)=(\tau x)^{\sim}\oplus \tau y$. \\
Similarly, the second equality of $(E_3)$ can be written as follows: \\
$\hspace*{3cm}$ $\sigma ((\sigma y)^{-}\odot x)=(\sigma y)^{-}\odot \sigma x$, \\
$\hspace*{3cm}$ $\sigma(x^{\sim}\oplus (\sigma y)^{-\sim})^{-}=((\sigma x)^{\sim}\oplus (\sigma y)^{-\sim})^{-}$, \\
$\hspace*{3cm}$ $\sigma(x^{\sim}\oplus \sigma y)^{-}=((\sigma x)^{\sim}\oplus \sigma y)^{-}$. \\
Replacing $x$ by $x^{-}$ and $y$ by $y^{-}$ we get: \\
$\hspace*{3cm}$ $\sigma(x\oplus \sigma y^{-})^{-}=((\sigma x^{-})^{\sim}\oplus \sigma y^{-})^{-}$, \\
$\hspace*{3cm}$ $\sigma(x\oplus (\tau y)^{-})^{-}=(\tau x\oplus (\tau y)^{-})^{-}$, \\
$\hspace*{3cm}$ $(\tau(x\oplus (\tau y)^{-})^{-}=(\tau x\oplus (\tau y)^{-})^{-}$, \\
$\hspace*{3cm}$ $\tau(x\oplus (\tau y)^{-})=\tau x\oplus (\tau y)^{-}$. \\
It follows that $(E_3)$ implies $(U_3)$. \\
The second equality of $(E_4)$ can be written successively: \\
$\hspace*{3cm}$ $\sigma (\sigma x\odot y)=\sigma x\odot \sigma y$, \\
$\hspace*{3cm}$ $\sigma (\sigma y\odot x)=\sigma y\odot \sigma x$, \\
$\hspace*{3cm}$ $\sigma(x^{\sim}\oplus (\sigma y)^{\sim})^{-}=((\sigma x)^{\sim}\oplus (\sigma y)^{\sim})^{-}$. \\
Replacing $x$ by $x^{-}$ and $y$ by $y^{-}$ we get: \\
$\hspace*{3cm}$ $\sigma (x\oplus (\sigma y^{-})^{\sim})^{-}=((\sigma x^{-})^{\sim}\oplus (\sigma y^{-})^{\sim})^{-}$, \\
$\hspace*{3cm}$ $\sigma (x\oplus \tau y)^{-}=(\tau x\oplus \tau y)^{-}$, \\
$\hspace*{3cm}$ $(\tau (x\oplus \tau y))^{-}=(\tau x\oplus \tau y)^{-}$, \\
$\hspace*{3cm}$ $\tau (x\oplus \tau y)=\tau x\oplus \tau y$. \\
Hence $(E_4)$ implies $(U_4)$. 
The first identity of $(E_5)$ can be written as follows: \\
$\hspace*{3cm}$ $\sigma((x^{-}\odot x^{-})^{\sim})=((\sigma x)^{-}\odot (\sigma x)^{-})^{\sim}$, \\
$\hspace*{3cm}$ $(\tau (x^{-}\odot x^{-})^{\sim}=(\tau x^{-}\odot \tau x^{-})^{\sim}$, \\
$\hspace*{3cm}$ $\tau (x^{-}\odot x^{-})=\tau x^{-}\odot \tau x^{-}$, \\
and replacing $x$ by $x^{\sim}$: \\
$\hspace*{3cm}$ $\tau (x\odot x)=\tau x\odot \tau x)$, \\
$\hspace*{3cm}$ $\tau((x^{-}\oplus x^{-})^{\sim})=((\tau x)^{-}\oplus (\tau x)^{-})^{\sim}$, \\
that is the first identity from $(U_5)$. Similarly we get the second identity from $(U_5)$. \\ 
From $(E_6)$ we get successively: \\
$\hspace*{3cm}$ $\sigma(x\odot x)=\sigma x\odot \sigma x$, \\
$\hspace*{3cm}$ $\sigma((x^{-}\oplus x^{-})^{\sim})=((\sigma x)^{-}\oplus (\sigma x)^{-})^{\sim}$, \\
$\hspace*{3cm}$ $(\tau(x^{-}\oplus x^{-}))^{\sim}=(\tau x^{-}\oplus \tau x^{-})^{\sim}$, \\
$\hspace*{3cm}$ $\tau(x^{-}\oplus x^{-})=\tau x^{-}\oplus \tau x^{-}$. \\
Replacing $x$ by $x^{\sim}$ we get $(U_6)$. 
Applying Theorem \ref{minv-psBE-200} it follows that $(A,\forall,\exists)$ is a monadic bound commutative 
pseudo BE-algebra. 
\end{proof}

The following result is a consequence of Theorems \ref{minv-psBE-200} and \ref{minv-psBE-210}. 

\begin{theorem} \label{minv-psBE-220}
In a monadic bounded commutative pseudo BE-algebra $(A,\exists,\forall)$ there is a one-to-one correspondence 
between the universal and existential quantifiers. 
If $\forall$ is a universal quantifier on $A$, then the corresponding existential quantifier $\exists$ on $A$ is 
defined by $\exists x=(\forall x^{-})^{\sim}=(\forall x^{\sim})^{-}$, for all $x\in A$.  
Dually, given an existential quantifier $\exists$ on $A$, the corresponding universal quantifier is defined as 
$\forall x=(\exists x^{-})^{\sim}=(\exists x^{\sim})^{-}$, for all $x\in A$.
\end{theorem}

\begin{remark} \label{minv-psBE-230}
The pseudo BE-algebra $(A,\ra,\rs,0,1)$ from Theorems \ref{minv-psBE-200} and \ref{minv-psBE-210} need not be commutative. These results are still valid for the case of involutive pseudo BE(A)-algebras. 
Indeed, consider the structure $(A,\ra,\rs,1)$, where the operations $\ra$ and $\rs$ on $A=\{1,a,b,c,d,e\}$ 
are defined as follows:
\[
\begin{array}{c|cccccc}
\ra & 1 & a & b & c & d & e \\ \hline
1 & 1 & a & b & c & d & e \\
a & 1 & 1 & d & 1 & 1 & d \\
b & 1 & c & 1 & 1 & 1 & c \\
c & 1 & a & d & 1 & d & a \\
d & 1 & c & b & c & 1 & b \\
e & 1 & 1 & 1 & 1 & 1 & 1 
\end{array}
\hspace{10mm}
\begin{array}{c|cccccc}
\rs & 1 & a & b & c & d & e \\ \hline
1 & 1 & a & b & c & d & e \\
a & 1 & 1 & c & 1 & 1 & c \\
b & 1 & d & 1 & 1 & 1 & d \\
c & 1 & d & b & 1 & d & b \\
d & 1 & a & c & c & 1 & a \\
e & 1 & 1 & 1 & 1 & 1 & 1
\end{array}
.
\]

Then $(A,\ra,\rs,e,1)$ is an involutive pseudo BCK-algebra, so it is an involutive pseudo BE(A)-algebra, 
but not commutative (\cite{Ciu30}).  
Consider the maps $\tau, \sigma, \forall, \exists:A\longrightarrow A$ defined by the following table:
\[
\begin{array}{c|cccccc}
               x    & 1 & a & b & c  & d & e \\ \hline
\tau x              & 1 & e & e & e & e & e \\
(\tau x^{-})^{\sim} & 1 & 1 & 1 & 1 & 1 & e \\
(\tau x^{\sim})^{-} & 1 & 1 & 1 & 1 & 1 & e \\
\sigma x            & 1 & 1 & 1 & 1 & 1 & e \\
\forall x=\tau x    & 1 & e & e & e & e & e \\
\exists x=\sigma x  & 1 & 1 & 1 & 1 & 1 & e
\end{array}
.   
\]
One can easily check that the operators $\tau$ and $\sigma$ satisfy the conditions from Theorems \ref{minv-psBE-200} and \ref{minv-psBE-210}, respectively, hence $(\exists,\forall)$ is a monadic involutive pseudo BE-algebra. 
\end{remark}

$\vspace*{5mm}$

\section{Monadic classes of pseudo BE-algebras}

We investigate properties of particular classes of monadic pseudo BCK-algebras, such as monadic pseudo BCK-algebras, 
monadic pseudo BCK-semilattices, monadic pseudo BCK(pP)-algebras, monadic pseudo-hoops. 
In the case of pseudo BCK-algebras it is proved that, if two universal quantifiers have the same image, 
then the two quantifiers coincide. 
We also show that the universal and existential quantifiers on a bounded commutative pseudo BE-algebra are also 
quantifiers on the corresponding pseudo MV-algebras. \\

We recall that, according to Corollary \ref{mpsBE-25-10}, if $(A,\exists,\forall)$ is a monadic pseudo BCK-algebra, 
then $\exists$ and $\forall$ are isotone. 

\begin{proposition} \label{mpsBCK-10} Let $(A,\exists,\forall)$ be a monadic pseudo BCK-algebra. Then the following hold, 
for all $x, y\in A:$ \\
$(1)$ $\forall(x\ra y)\rs (\forall x\ra \forall y)=1$, $\forall(x\rs y)\ra (\forall x\rs \forall y)=1;$ \\
$(2)$ $\exists(x\ra \exists y)\rs \forall(\forall x\ra \exists y)=1$, 
      $\exists(x\rs \exists y)\ra \forall(\forall x\rs \exists y)=1;$ \\
$(3)$ $\exists(\exists x\ra y)\rs (\exists x\ra \exists y)=1$, 
      $\exists(\exists x\rs y)\ra (\exists x\rs \exists y)=1;$ \\
$(4)$ $\forall(x\ra y)\rs (\exists x\ra \exists y)=1$, $\forall(x\rs y)\ra (\exists x\rs \exists y)=1$.     
\end{proposition}
\begin{proof}
$(1)$ By $(M_2)$ and Proposition \ref{psBE-20}$(2)$ we have $\forall(x\ra y)\le x\ra y\le \forall x\ra y$, that is 
$\forall(x\ra y)\rs (\forall x\ra y)=1$. Using $(M_5)$ and $(M_4)$ we get \\
$\hspace*{2cm}$ $1=\forall 1=\forall(\forall(x\ra y)\rs (\forall x\ra y))=
                   \forall(\exists\forall(x\ra y)\rs (\forall x\ra y))$ \\
$\hspace*{2.3cm}$ $=\exists\forall(x\ra y)\rs \forall(\forall x\ra y)=
                \forall(x\ra y)\rs \forall(\exists\forall x\ra y)$ \\
$\hspace*{2.3cm}$ $=\forall(x\ra y)\rs (\exists\forall x\ra \forall y)=\forall(x\ra y)\rs (\forall x\ra \forall y)$. \\
Similarly, $\forall(x\rs y)\ra (\forall x\rs \forall y)=1$. \\
$(2)$ Using $(M_2)$ and Proposition \ref{psBE-20}$(2)$ we get 
$x\ra \exists y\le \forall x\ra \exists y$, hence 
$(x\ra \exists y)\rs (\forall x\ra \exists y)=1$. Applying $(M_5)$ and $(M_3)$, it follows that: \\
$\hspace*{2cm}$  $1=\forall 1=\forall((x\ra \exists y)\rs (\forall x\ra \exists y))
                   =\forall((x\ra \exists y)\rs (\exists\forall x\ra \exists y))$ \\
$\hspace*{2.3cm}$ $=\forall((x\ra \exists y)\rs \forall(\forall x\ra \exists y)) 
                   =\forall((x\ra \exists y)\rs \exists\forall(\forall x\ra \exists y))$ \\
$\hspace*{2.3cm}$ $=\exists(x\ra \exists y)\rs \exists\forall(\forall x\ra \exists y)  
                   =\exists(x\ra \exists y)\rs \forall(\forall x\ra \exists y)$. \\
Similarly, $\exists(x\rs \exists y)\ra \forall(\forall x\rs \exists y)=1$. \\
$(3)$ Using $(M_1)$ and Proposition \ref{psBE-20}$(3)$ we get $\exists x\ra y\le \exists x\ra \exists y$, 
hence $(\exists x\ra y)\rs (\exists x\ra \exists y)=1$. 
It follows that $\forall((\exists x\ra y)\rs (\exists x\ra \exists y))=1$, so by $(M_3)$ we have 
$\forall((\exists x\ra y)\rs \forall(x\ra \exists y))=1$. 
Applying $(M_5)$ and $(M_3)$, we get: \\
$\hspace*{2cm}$   $1=\forall((\exists x\ra y)\rs \forall(x\ra \exists y))
                    =\forall((\exists x\ra y)\rs \exists\forall(x\ra \exists y))$ \\
$\hspace*{2.3cm}$ $=\exists(\exists x\ra y)\rs \exists\forall(x\ra \exists y)
                   =\exists(\exists x\ra y)\rs \forall(x\ra \exists y)$ \\
$\hspace*{2.3cm}$ $=\exists(\exists x\ra y)\rs (\exists x\ra \exists y)$. \\
Similarly, $\exists(\exists x\rs y)\ra (\exists x\rs \exists y)=1$. \\
$(4)$ Using $(M_2)$, $(M_1)$ and Proposition \ref{psBE-20}$(3)$ we have 
$\forall(x\ra y)\le x\ra y\le x\ra \exists y$, so $\forall(x\ra y)\rs (x\ra \exists y)=1$. 
By $(M_5)$, $(M_4)$,$(M_3)$, it follows that: \\
$\hspace*{2cm}$   $1=\forall 1=\forall(\forall(x\ra y)\rs (x\ra \exists y))
                    =\forall(\exists\forall(x\ra y)\rs (x\ra \exists y))$ \\
$\hspace*{2.3cm}$ $=\exists\forall(x\ra y)\rs \forall(x\ra \exists y)                  
                   =\forall(x\ra y)\rs \forall(x\ra \exists y)$ \\
$\hspace*{2.3cm}$ $=\forall(x\ra y)\rs (\exists x\ra \exists y)$. \\
Similarly, $\forall(x\rs y)\ra (\exists x\rs \exists y)=1$.  
\end{proof}

\begin{proposition} \label{mpsBCK-15} In any monadic pseudo BCK-algebra $A$ the following hold, for all $x, y\in A:$ \\
$(1)$ $\exists(\exists x\ra \exists y)=\exists x\ra \exists y$, 
      $\exists(\exists x\rs \exists y)=\exists x\rs \exists y;$ \\
$(2)$ $\exists(\forall x\ra \forall y)=\forall x\ra \forall y$,
      $\exists(\forall x\rs \forall y)=\forall x\rs \forall y;$ \\
$(3)$ $\forall(\exists x\ra \exists y)=\exists x\ra \exists y$, 
      $\forall(\exists x\rs \exists y)=\exists x\rs \exists y;$ \\
$(4)$ $\forall(\forall x\ra \forall y)=\forall x\ra \forall y$,
      $\forall(\forall x\rs \forall y)=\forall x\rs \forall y$.  
\end{proposition}
\begin{proof}
$(1)$ By Proposition \ref{mpsBE-20}$(14)$ we have $\exists(\exists x\ra \exists y)\le \exists x\ra \exists y$, 
and using $(M_1)$, $\exists x\ra \exists y \le \exists(\exists x\ra \exists y)$. 
Hence $\exists(\exists x\ra \exists y)=\exists x\ra \exists y$. Similarly, 
$\exists(\exists x\rs \exists y)=\exists x\rs \exists y$. \\ 
$(2)$ Applying $(M_5)$ and $(1)$ we get  
$\exists(\forall x\ra \forall y)=\exists(\exists\forall x\ra \exists\forall y) 
                 =\exists\exists\forall x\ra \exists\exists\forall y=\forall x\ra \forall y$. 
Similarly, $\exists(\forall x\rs \forall y)=\forall x\rs \forall y$. \\
$(3)$ It follows by $(1)$ and Proposition \ref{mpsBE-20}$(4)$. \\
$(4)$ It follows by $(2)$ and Proposition \ref{mpsBE-20}$(4)$. 
\end{proof}

Given the universal quantifiers $\forall_1$, $\forall_2$ on a pseudo BE-algebra $A$, we define 
$\forall_1\le \forall_2$ by $\forall_1 x\le \forall_2 x$ and 
$(\forall_1\circ \forall_2)(x)=\forall_1(\forall_2 x)$, for all $x\in A$. 
Similarly for the case of existential quantifiers. We will use the notations $\forall_1\forall_2$ and 
$\exists_1\exists_2$ instead of $\forall_1\circ \forall_2$ and $\exists_1\circ \exists_2$, respectively. 
 
\begin{theorem} \label{mpsBCK-20} Let $A$ be a pseudo BCK-algebra and let 
$(\exists_1,\forall_1), (\exists_2,\forall_2)\in \mathcal{MOP}(A)$. Then the following hold: \\
$(1)$ $\forall_1\le \forall_2$ iff $\forall_1\forall_2=\forall_1;$ \\ 
$(2)$ $\exists_1\ge \exists_2$ iff $\exists_1\exists_2=\exists_1;$ \\
$(3)$ $A_{\exists_1\forall_1}=A_{\exists_2\forall_2}$ implies $\forall_1=\forall_2;$ \\
$(4)$ $\Img(\forall_1)=\Img(\forall_2)$ implies $\forall_1=\forall_2$.
\end{theorem}
\begin{proof}
$(1)$ Assume that $\forall_1 \le \forall_2$, that is $\forall_1 x\le \forall_2 x$, for all $x\in A$. Then we have 
$\forall_1 x=\forall_1\forall_1 x\le \forall_1\forall_2 x$, so $\forall_1\le \forall_1\forall_2$. 
Moreover, 
$\forall_1\forall_2 x\le \forall_2\forall_2 x=\forall_2 x\le x$, hence 
$\forall_1\forall_1\forall_2 x\le \forall_1 x$, so $\forall_1\forall_2 x\le \forall_1 x$, that is 
$\forall_1\forall_2\le \forall_1$. Hence, $\forall_1\forall_2=\forall_1$. 
Conversely, if $\forall_1\forall_2=\forall_1$, we have $\forall_1 x=\forall_1\forall_2 x\le \forall_2 x$, 
for all $x\in A$, hence $\forall_1\le \forall_2$. \\ 
$(2)$ If $\exists_1\ge \exists_2$, then $\exists_1 x\ge \exists_2 x$, fot all $x\in A$. 
Hence, $\exists_1 x=\exists_1\exists_1 x\ge \exists_1\exists_2 x$, so $\exists_1\ge \exists_1\exists_2$. 
On the other hand, $\exists_1\exists_2 x\ge \exists_2\exists_2 x=\exists_2 x\ge x$, thus 
$\exists_1\exists_1\exists_2 x\ge \exists_1 x$, so that $\exists_1\exists_2 x\ge \exists_1 x$, for all $x\in A$, 
that is $\exists_1\exists_2 \ge \exists_1$. It follows that $\exists_1\exists_2 =\exists_1$. 
Conversely, if $\exists_1\exists_2 \ge \exists_1$, then $\exists_1x=\exists_1\exists_2 x\ge \exists_2 x$, for all 
$x\in A$, hence $\exists_1\ge \exists_2$. \\ 
$(3)$ From $\forall_1\forall_1 x=\forall_1 x$ we get $\forall_1 x\in A_{\exists_1\forall_1}=A_{\exists_2\forall_2}$, 
hence $\forall_2\forall_1 x=\forall_1 x$, for all $x\in A$. Thus $\forall_2\forall_1=\forall_1$, and similarly 
$\forall_1\forall_2=\forall_2$. 
Since $\forall_1 x\le x$, $\forall_2 x\le x$, for all $x\in A$, we have 
$\forall_1 x=\forall_2\forall_1 x\le \forall_2 x$. 
Similarly, $\forall_2 x=\forall_1\forall_2 x\le \forall_1 x$. 
It follows that $\forall_1=\forall_2$. \\
$(4)$ It follows by $(3)$ and Proposition \ref{mpsBE-90}$(1)$. 
\end{proof}

\begin{lemma} \label{mpsBCK-30} Let $A$ be a pseudo BCK-algebra and let 
$(\exists_1,\forall_1), (\exists_2,\forall_2)\in \mathcal{MOP}(A)$. Then the following are equivalent: \\
$(a)$ $\exists_1\exists_2=\exists_2\exists_1$ and $\forall_1\forall_2=\forall_2\forall_1;$ \\
$(b)$ $\exists_1\exists_2, \exists_2\exists_1$ satisfy $(M_1)$ and  
      $\forall_1\forall_2, \forall_2\forall_1$ satisfy $(M_2);$ \\
$(c)$ $\exists_1\exists_2\exists_1\exists_2=\exists_1\exists_2$,      
      $\exists_2\exists_1\exists_2\exists_1=\exists_2\exists_1$,  
      $\forall_1\forall_2\forall_1\forall_2=\forall_1\forall_2$ and 
      $\forall_2\forall_1\forall_2\forall_1=\forall_2\forall_1$. 
\end{lemma}
\begin{proof}
$(a)\Rightarrow (b)$ For any $x\in A$ we have: $\exists_1\exists_2 x\ge \exists_2 x\ge x$, 
$\exists_2\exists_1 x\ge \exists_1 x\ge x$, $\forall_1\forall_2 x\le \forall_1 x\le x$ and 
$\forall_2\forall_1 x\le \forall_2 x\le x$, hence $\exists_1\exists_2, \exists_2\exists_1$ satisfy $(M_1)$, while   
      $\forall_1\forall_2, \forall_2\forall_1$ satisfy $(M_2)$. \\
$(b)\Rightarrow (c)$ First of all, we can see that, in the proof of Theorem \ref{mpsBCK-20}$(1)$,$(2)$ it is 
sufficient that $\forall_1$, $\forall_2$ satisfy $(M_2)$ and $\exists_1$, $\exists_2$ satisfy $(M_1)$. \\
Since $\exists_1\exists_2\exists_1\exists_2\ge \exists_1\exists_2$ and 
and $\exists_2\exists_1\exists_2\exists_1\ge \exists_2\exists_1$, then by Theorem \ref{mpsBCK-20}$(2)$ we get 
$\exists_1\exists_2\exists_1\exists_2=\exists_1\exists_2$ and $\exists_2\exists_1\exists_2\exists_1=\exists_2\exists_1$. 
Similarly, from $\forall_1\forall_2\forall_1\forall_2\le \forall_1\forall_2$ and 
$\forall_2\forall_1\forall_2\forall_1\le \forall_2\forall_1$, by Theorem \ref{mpsBCK-20}$(1)$ we have 
$\forall_1\forall_2\forall_1\forall_2=\forall_1\forall_2$ and 
$\forall_2\forall_1\forall_2\forall_1=\forall_2\forall_1$. \\
$(c)\Rightarrow (a)$ Applying $(b)$, for any $x\in A$ we have: \\
$\hspace*{2cm}$ $\exists_1\exists_2 x=\exists_1\exists_2\exists_1\exists_2 x\ge \exists_2\exists_1\exists_2 x \ge 
                 \exists_2\exists_1 x$ and \\                    
$\hspace*{2cm}$ $\exists_2\exists_1 x=\exists_2\exists_1\exists_2\exists_1 x\ge \exists_1\exists_2\exists_1 x \ge 
                 \exists_1\exists_2 x$. \\
$\hspace*{2cm}$ $\forall_1\forall_2 x=\forall_1\forall_2\forall_1\forall_2 x\le \forall_2\forall_1\forall_2 x\le  
                 \forall_2\forall_1 x$ and \\
$\hspace*{2cm}$ $\forall_2\forall_1 x=\forall_2\forall_1\forall_2\forall_1 x\le \forall_1\forall_2\forall_1 x\le  
                 \forall_1\forall_2 x$. \\                
Hence $\exists_1\exists_2=\exists_2\exists_1$ and $\forall_1\forall_2=\forall_2\forall_1$. 
\end{proof}

\begin{theorem} \label{mpsBCK-40} Let $A$ be a pseudo BCK-algebra and let 
$(\exists_1,\forall_1), (\exists_2,\forall_2)\in \mathcal{MOP}(A)$. Then 
$(\exists_1\exists_2,\forall_1\forall_2),(\exists_2\exists_1,\forall_2\forall_1)\in \mathcal{MOP}(A)$ 
if and only if $\exists_1\exists_2=\exists_2\exists_1$ and $\forall_1\forall_2=\forall_2\forall_1$. 
\end{theorem}
\begin{proof}
Let $(\exists_1,\forall_1), (\exists_2,\forall_2)\in \mathcal{MOP}(A)$ such that 
$(\exists_1\exists_2,\forall_1\forall_2),(\exists_2\exists_1,\forall_2\forall_1)\in \mathcal{MOP}(A)$. 
Applying Lemma \ref{mpsBCK-30}, it follows that 
$\exists_1\exists_2=\exists_2\exists_1$ and $\forall_1\forall_2=\forall_2\forall_1$. 
Conversely, assume that $\exists_1\exists_2=\exists_2\exists_1$ and $\forall_1\forall_2=\forall_2\forall_1$. 
Since $x\le \exists_1 x\le \exists_1\exists_2 x$, then $\exists_1\exists_2$ satisfies $(M_1)$. 
From $\forall_1\forall_2 x\le \forall_2 x\le x$, it follows that $\forall_1\forall_2$ verifies $(M_2)$. 
For any $x, y\in A$, applying axioms $(M_3)$-$(M_5)$ for $(\exists_1,\forall_1)$ and $(\exists_2,\forall_2)$ 
we get: \\
$\hspace*{2cm}$ $\forall_1\forall_2(x\ra \exists_1\exists_2 y)=\forall_1\forall_2(x\ra \exists_2\exists_1 y) 
                 =\forall_1(\exists_2 x\ra \exists_2\exists_1 y)$ \\
$\hspace*{5cm}$ $=\forall_1(\exists_2 x\ra \exists_1\exists_2 y)  
                 =\exists_1\exists_2 x\ra \exists_1\exists_2 y$. \\
$\hspace*{2cm}$ $\forall_1\forall_2(x\rs \exists_1\exists_2 y)=\forall_1\forall_2(x\rs \exists_2\exists_1 y) 
                 =\forall_1(\exists_2 x\rs \exists_2\exists_1 y)$ \\
$\hspace*{5cm}$ $=\forall_1(\exists_2 x\rs \exists_1\exists_2 y)  
                 =\exists_1\exists_2 x\rs \exists_1\exists_2 y$. \\
$\hspace*{2cm}$ $\forall_1\forall_2(\exists_1\exists_2 x\ra y)=\forall_1\forall_2(\exists_2\exists_1 x\ra y) 
                 =\forall_1(\exists_2\exists_1 x\ra \forall_2 y)$ \\
$\hspace*{5cm}$ $=\forall_1(\exists_1\exists_2 x\ra \forall_2 y)  
                 =\exists_1\exists_2 x\ra \forall_1\forall_2 y$. \\
$\hspace*{2cm}$ $\forall_1\forall_2(\exists_1\exists_2 x\rs y)=\forall_1\forall_2(\exists_2\exists_1 x\rs y) 
                 =\forall_1(\exists_2\exists_1 x\rs \forall_2 y)$ \\
$\hspace*{5cm}$ $=\forall_1(\exists_1\exists_2 x\rs \forall_2 y)  
                 =\exists_1\exists_2 x\rs \forall_1\forall_2 y$. \\
$\hspace*{2cm}$ $\exists_1\exists_2\forall_1\forall_2 x=\exists_2\exists_1\forall_1\forall_2 x=
                 \exists_2\forall_1\forall_2 x=\exists_2\forall_2\forall_1 x=\forall_2\forall_1 x$. \\
Hence, axioms $(M_3)$, $(M_4)$ and $(M_5)$ are satisfied. \\
We conclude that $(\exists_1\exists_2,\forall_1\forall_2)\in \mathcal{MOP}(A)$. 
Similarly, $(\exists_2\exists_1,\forall_2\forall_1)\in \mathcal{MOP}(A)$. 
\end{proof}

\begin{example} \label{mpsBE-45} Consider the pseudo BE-algebra $A$ and the monadic operators 
$(\exists_2,\forall_2)$, $(\exists_3,\forall_3)$ from Example \ref{mpsBE-10-30}. One can easily check that $\exists_2\exists_3=\exists_3\exists_2=\exists_4$ and $\forall_2\forall_3=\forall_3\forall_2=\forall_4$. 
Hence $(\exists_2\exists_3,\forall_2\forall_3)=(\exists_3\exists_2,\forall_3\forall_2)=(\exists_4,\forall_4)\in 
\mathcal{MOP}(A)$.
\end{example}


\begin{proposition} \label{mpsBE-60} If $(A,\exists,\forall)$ is a monadic pseudo BCK-meet-semilattice, then the 
following hold, for all $x, y\in A:$ \\
$(1)$ $\exists(\exists x\wedge \exists y)=\exists x\wedge \exists y;$ \\
$(2)$ $\forall(\exists x\wedge \exists y)=\exists x\wedge \exists y;$ \\
$(3)$ $\forall(x\wedge y)=\forall x\wedge \forall y;$ \\
$(4)$ $\exists(x\wedge y)\le \exists x\wedge \exists y;$ \\
$(5)$ $\forall(x\wedge y)\le \exists x\wedge \exists y$. 
\end{proposition} 
\begin{proof}
$(1)$ By $(M_1)$, $\exists x\wedge \exists y\le \exists(\exists x\wedge \exists y)$. 
On the other hand, from $\exists x\wedge \exists y\le \exists y$, $\exists x\wedge \exists y\le \exists y$ we 
get $\exists(\exists x\wedge \exists y)\le \exists\exists x=\exists x$ and 
$\exists(\exists x\wedge \exists y)\le \exists\exists y=\exists y$. 
Hence $\exists(\exists x\wedge \exists y)\le \exists x\wedge \exists y$. 
It follows that $\exists(\exists x\wedge \exists y)=\exists x\wedge \exists y$. \\
$(2)$ It follows by $(1)$ and Proposition \ref{mpsBE-20}$(4)$. \\
$(3)$ Let $z\le \forall x\wedge \forall y$, so $z\le \forall x$ and $z\le \forall y$. 
It follows that $\exists z\le \exists\forall x=\forall x$ and $\exists z\le \exists\forall y=\forall y$. 
Hence $\exists z\le \forall x\wedge \forall y\le x\wedge y$, so 
$z\le \exists z=\forall\exists z\le \forall(x\wedge y)$. 
Taking $z:=\forall x\wedge \forall y$, we get $\forall x\wedge \forall y\le \forall (x\wedge y)$. 
On the other hand, since $x\wedge y\le x,y$, we have $\forall(x\wedge y)\le \forall x\wedge \forall y$. 
Thus $\forall(x\wedge y)= \forall x\wedge \forall y$. \\
$(4)$ From $x\le \exists x$ and $y\le \exists y$ we get $x\wedge y\le \exists x\wedge \exists y$. 
Using $(1)$ we have $\exists(x\wedge y)\le \exists(\exists x\wedge \exists y)=\exists x\wedge \exists y$. \\
$(5)$ Similarly as in $(4)$, $x\wedge y\le \exists x\wedge \exists y$, so by $(3)$, 
$\forall(x\wedge y)\le \forall(\exists x\wedge \exists y)=\forall\exists x\wedge \forall\exists x=
\exists x\wedge \exists y$. 
\end{proof}

\begin{proposition} \label{mpsBE-70} If $(A,\exists,\forall)$ is a monadic pseudo BCK-join-semilattice, then the 
following hold, for all $x, y\in A:$ \\
$(1)$ $\forall(\exists x\vee \exists y)=\exists x\vee \exists y;$ \\
$(2)$ $\exists(\exists x\vee \exists y)=\exists x\vee \exists y;$ \\
$(3)$ $\exists(x\vee y)=\exists x\vee \exists y;$ \\
$(4)$ $\forall x\vee \exists y\le \forall(x\vee \exists y)$. 
\end{proposition} 
\begin{proof}
$(1)$ From $\exists x\le \exists x\vee \exists y$ and $\exists y\le \exists x\vee \exists y$ we get 
$\exists x=\forall\exists x\le \forall(\exists x\vee \exists y)$ and 
$\exists y=\forall\exists y\le \forall(\exists x\vee \exists y)$, 
so $\exists x\vee \exists y\le \forall(\exists x\vee \exists y)$. 
Since by $(M_2)$, $\forall(\exists x\vee \exists y)\le \exists x\vee \exists y$, it follows that 
$\forall(\exists x\vee \exists y)=\exists x\vee \exists y$. \\
$(2)$ It is a consequence of $(1)$ and Proposition \ref{mpsBE-20}$(4)$. \\
$(3)$ From $x\le x\vee y$ and $y\le x\vee y$ we have $\exists x\le \exists(x\vee y)$ and 
$\exists y\le \exists(x\vee y)$, so $\exists x \vee \exists y\le \exists(x\vee y)$. 
On the other hand, from $x\le \exists x$ and $y\le \exists x$ we get $x\vee y\le \exists x\vee \exists y$, and 
using $(8)$ and $(18)$ we have $\exists(x\vee y)\le \exists(\exists x\vee \exists y)=\exists x\vee \exists y$. 
Hence $\exists(x\vee y)=\exists x\vee \exists y$. \\
$(4)$ From $\forall x\le x$, we get $\forall x\vee \exists y\le x\vee \exists y$, and using $(M_5)$ we have 
$\exists\forall x\vee \exists y\le x\vee \exists y$. It follows that 
$\forall(\exists\forall x\vee \exists y)\le \forall(x\vee \exists y)$, and applying $(1)$ we get 
$\forall\exists\forall x\vee \forall\exists y\le \forall(x\vee \exists y)$. 
Finally, since $\exists\forall x=\forall x$, $\forall\forall x=\forall x$ and $\forall\exists x=\exists x$, 
we have $\forall x\vee \exists y\le \forall(x\vee \exists y)$. 
\end{proof}

\begin{proposition} \label{mpp-BCK-20} In any monadic pseudo BCK(pP)-algebra $A$ the following hold, for all 
$x, y\in A:$ \\ 
$(1)$ $\exists(\exists x\odot \exists y)=\exists x\odot \exists y;$ \\
$(2)$ $\exists(x\odot y)\le \exists x\odot \exists y;$ \\
$(3)$ $\forall(\forall x\odot \forall y)=\forall x\odot \forall y;$ \\
$(4)$ $\forall x\odot \forall y\le \forall(x\odot y);$ \\ 
$(5)$ $\forall x\odot \forall y\le \exists(x\odot y)$. 
\end{proposition}
\begin{proof}
$(1)$ From $\exists x\odot \exists y\le \exists x\odot \exists y$ we have 
$\exists x\le \exists y \ra \exists x\odot \exists y$, and using Proposition \ref{mpsBE-20}$(3)$ 
and $(M_4)$ we get 
$\exists x=\forall\exists x\le \forall(\exists y\ra \exists x\odot \exists y)=
\exists y\ra \forall(\exists x\odot \exists y)$. 
It follows that $\exists x\odot \exists y\le \forall(\exists x\odot \exists y)$. 
On the other hand, by $(M_2)$, $\forall(\exists x\odot \exists y)\le \exists x\odot \exists y$, hence             
$\forall(\exists x\odot \exists y)=\exists x\odot \exists y$. Applying Proposition \ref{mpsBE-20}$(4)$, we get          $\exists(\exists x\odot \exists y)=\exists x\odot \exists y$. \\ 
$(2)$ From $x\le \exists x$ and $y\le \exists y$ we get $x\odot y\le \exists x\odot \exists y$, and using $(1)$ 
it follows that $\exists(x\odot y)\le \exists(\exists x\odot \exists y)=\exists x\odot \exists y$. \\
$(3)$ By $(2)$ and $(M_5)$ we have 
$\exists(\forall x\odot \forall y)\le \exists\forall x\odot \exists\forall y=\forall x\odot \forall y$. 
On the other hand, by $(M_1)$, $\forall x\odot \forall y\le \exists(\forall x\odot \forall y)$, hence 
$\exists(\forall x\odot \forall y)=\forall x\odot \forall y$. Applying Proposition \ref{mpsBE-20}$(4)$ we get 
$\forall(\forall x\odot \forall y)=\forall x\odot \forall y$. \\
$(4)$ From $\forall x\le x$ and $\forall y\le y$ we have $\forall x\odot \forall y\le x\odot y$, hence, 
by $(3)$, we get $\forall x\odot \forall y=\forall(\forall x\odot \forall y)\le \forall(x\odot y)$. \\
$(5)$ By $(M_1)$ and $(M_2)$ we have 
$1=\forall x\odot \forall y\ra x\odot y\le \forall x\odot \forall y\ra \exists(x\odot y)$, so that 
$\forall x\odot \forall y\ra \exists(x\odot y)=1$, hence $\forall x\odot \forall y\le\exists(x\odot y)$. 
\end{proof}

By the next result we extend to the case of monadic pseudo BCK(pP)-algebras some results proved in \cite{Zah1} 
for monadic bounded commutative BE-algebras (in fact, monadic bounded commutative dual BCK-algebras). 

\begin{proposition} \label{mpp-BCK-20-10} In any monadic pseudo BCK(pP)-algebra $A$ the following hold, for all 
$x, y\in A:$ \\ 
$(1)$ $\forall x\odot \exists y\ra \exists(x\odot y)=1$ and $\exists x \odot \forall y\rs \exists(x\odot y)=1;$ \\ 
$(2)$ $\exists x\odot \exists y\ra \exists(\exists x\odot y)=1$ and 
      $\exists x\odot \exists y\rs \exists(x\odot \exists y)=1;$ \\
$(3)$ $\exists(\exists x\odot y)=\exists x\odot \exists y=\exists(x\odot \exists y);$ \\
$(4)$ $\exists(x\odot \forall y)=\exists x\odot \forall y$ and $\exists(\forall x\odot y)=\forall x\odot \exists y$.  
\end{proposition}
\begin{proof}
$(1)$ The proof uses the fact that, by $(M_2)$ and Propositions \ref{psBE-25-10}$(3)$ and \ref{psBE-20}$(2)$, 
we have $\forall x\odot y\le x\odot y$, so $\forall x\odot y\ra z \ge x\odot y\ra z$, 
for all $x, y, \textit{}z\in A$. Then we have: \\
$\hspace*{1cm}$   $\forall x\odot \exists y\ra \exists(x\odot y)=\forall x\ra (\exists y\ra \exists(x\odot y))$ 
                                                (by Prop. \ref{psBE-25-10}$(5)$) \\
$\hspace*{4.5cm}$ $=\forall x\ra \forall(y\ra \exists(x\odot y))$ (by $(M_3))$ \\
$\hspace*{4.5cm}$ $=\exists\forall x\ra \forall(y\ra \exists(x\odot y))$ (by $(M_5)$) \\
$\hspace*{4.5cm}$ $=\forall(\exists\forall x\ra (y\ra \exists(x\odot y)))$ (by $(M_4)$) \\
$\hspace*{4.5cm}$ $=\forall(\forall x\ra (y\ra \exists(x\odot y)))$ (by $(M_5)$) \\
$\hspace*{4.5cm}$ $=\forall(\forall x\odot y\ra \exists(x\odot y))$ (by Prop. \ref{psBE-25-10}$(5)$) \\ 
$\hspace*{4.5cm}$ $\ge \forall(x\odot y\ra \exists(x\odot y))$ (by Prop. \ref{psBE-25-10}$(3)$, \ref{psBE-20}$(2)$) \\ 
$\hspace*{4.5cm}$ $=\forall 1=1$ (by $(M_2)$). \\
Hence, $\forall x\odot \exists y\ra \exists(x\odot y)=1$. Similarly, we get: \\
$\hspace*{1cm}$   $\exists x\odot \forall y\ra \exists(x\odot y)=\forall y\rs (\exists x\rs \exists(x\odot y))$ 
                                                (by Prop. \ref{psBE-25-10}$(5)$) \\
$\hspace*{4.5cm}$ $=\forall y\rs \forall(x\rs \exists(x\odot y))$ (by $(M_3))$ \\
$\hspace*{4.5cm}$ $=\exists\forall y\rs \forall(x\rs \exists(x\odot y))$ (by $(M_5)$) \\
$\hspace*{4.5cm}$ $=\forall(\exists\forall y\ra (x\ra \exists(x\odot y)))$ (by $(M_4)$) \\
$\hspace*{4.5cm}$ $=\forall(\forall y\rs (x\rs \exists(x\odot y)))$ (by $(M_5)$) \\
$\hspace*{4.5cm}$ $=\forall(\forall x\odot y\rs \exists(x\odot y))$ (by Prop. \ref{psBE-25-10}$(5)$) \\ 
$\hspace*{4.5cm}$ $\ge \forall(x\odot y\rs \exists(x\odot y))$ (by Prop. \ref{psBE-25-10}$(3)$, \ref{psBE-20}$(2)$) \\ 
$\hspace*{4.5cm}$ $=\forall 1=1$ (by $(M_2)$). \\
It follows that $\exists x \odot \forall y\rs \exists(x\odot y)=1$. \\
$(2)$ Since  $\forall\exists x=\exists x$, replacing $x$ by $\exists x$ in the first identity of $(1)$, we get 
$\exists x\odot \exists y\ra \exists(\exists x\odot y)=1$. 
Similarly, replacing $y$ by $\exists y$ in the second identity of $(1)$, we get 
$\exists x\odot \exists y\rs \exists(x\odot \exists y)=1$. \\
$(3)$ Applying $(2)$ we have $\exists x\odot \exists y\le \exists(\exists x\odot y)$ and 
$\exists x\odot \exists y\le \exists(x\odot \exists y)$. 
On the other hand, from $\exists x\odot y\le \exists x\odot \exists y$ and 
$x\odot \exists y\le \exists x\odot \exists y$, applying Proposition \ref{mpp-BCK-20}$(1)$ we get 
$\exists(\exists x\odot y)\le \exists(\exists x\odot \exists y)=\exists x\odot \exists y$ and 
$\exists(x\odot \exists y)\le \exists(\exists x\odot \exists y)=\exists x\odot \exists y$.  
We conclude that $\exists(\exists x\odot y)=\exists x\odot \exists y=\exists(x\odot \exists y)$. \\ 
$(4)$ From $(3)$, since $\exists\forall x=\forall x$, replacing $x$ by $\forall x$ in 
$\exists(\exists x\odot y)=\exists x\odot \exists y$, we get $\exists(\forall x\odot y)=\forall x\odot \exists y$. 
Similarly, replacing $y$ by $\forall y$ in $\exists(x\odot \exists y)=\exists x\odot \exists y$, it follows that 
$\exists(x\odot \forall y)=\exists x\odot \forall y$.   
\end{proof}

\begin{proposition} \label{mpp-BCK-30} In any monadic bounded commutative pseudo BCK-algebra $A$ the 
following hold, for all $x, y\in A:$ \\ 
$(1)$ $\forall(\forall x\oplus \forall y)=\forall x\oplus \forall y;$ \\
$(2)$ $\forall(x\oplus y)\ge \forall x\oplus \forall y$. 
\end{proposition}
\begin{proof}
Since $A$ is involutive, it is a pseudo BCK(pP)-algebra. \\ 
$(1)$ Applying Lemma \ref{minv-psBE-30-10} and Propositions \ref{mpsBE-50}$(7)$, \ref{mpp-BCK-20}$(3)$, we get: \\
$\hspace*{2cm}$   $\forall x\oplus \forall y=((\forall y)^{-}\odot (\forall x)^{-})^{\sim}=
                  (\forall(\forall y)^{-}\odot \forall(\forall x)^{-})^{\sim}$ \\
$\hspace*{3.4cm}$ $=(\forall((\forall y)^{-}\odot (\forall x)^{-}))^{\sim}=
                  \forall(\forall((\forall y)^{-}\odot (\forall x)^{-})^{\sim})$ \\
$\hspace*{3.4cm}$ $=\forall(\forall(\forall y)^{-}\odot \forall(\forall x)^{-})^{\sim}=
                  \forall((\forall y)^{-}\odot (\forall x)^{-})^{\sim}$ \\
$\hspace*{3.4cm}$ $=\forall(\forall x\oplus \forall y)$. \\
$(2)$ From $\forall x\le x$, $\forall y\le y$ we get $\forall x\oplus \forall y\le x\oplus y$, hence 
$\forall(\forall x\oplus \forall y)\le \forall(x\oplus y)$. Using $(1)$, it follows that 
$\forall x\oplus \forall y\le \forall(x\oplus y)$.
\end{proof}

Extending the notion of a monadic bounded hoop defined in \cite{Wang1}, we define the monadic pseudo-hoops.

\begin{definition} \label{mpsh-10} Let $(A,\odot,\ra,\rs,1,\exists,\forall)$ be an algebra of type $(2,2,2,0,1,1)$ 
such that $(A,\odot,\ra,\rs,1)$ is a pseudo-hoop. Then $(A,\exists,\forall)$ is called a \emph{monadic pseudo-hoop} 
if it satisfies conditions $(M_1)$-$(M_6)$. 
\end{definition}

\begin{proposition} \label{mpsh-20}
Any monadic pseudo BCK(pP)-algebra satisfying axiom $(M_6)$ is a monadic pseudo-hoop. 
\end{proposition}
\begin{proof}
According to \cite[Prop. 5.4]{Ciu14}, every pseudo-hoop $(A,\odot,\ra,\rs,1$ is a pseudo BCK(pP)-algebra. 
Since axioms $(M_1)$-$(M_5)$ hold in any monadic pseudo BCK-algebra, it follows that any monadic pseudo BCK(pP)-algebra satisfying axiom $(M_6)$ is a monadic pseudo-hoop. 
\end{proof}

\begin{proposition} \label{mpsh-30} 
Let $(A,\exists,\forall)$ be a monadic pseudo-hoop. The following hold, for all $x, y\in A:$ \\
$(1)$ $\forall(x\ra y)\odot x\le \exists x\wedge \exists y$ and 
      $\exists x\odot \forall(x\rs y)\le \exists x\wedge \exists y;$ \\
$(2)$ $\forall((x\ra y)\odot (x\ra y))\le (\exists x\ra \exists y)\odot (\exists x\ra \exists y)$ and 
      $\forall((x\rs y)\odot (x\rs y))\le (\exists x\rs \exists y)\odot (\exists x\rs \exists y)$. 
\end{proposition}
\begin{proof}
$(1)$ Applying Proposition \ref{mpsBCK-10}$(4)$ we have 
$\forall(x\ra y)\odot \exists x\le (\exists x\ra \exists y)\odot \exists x=\exists x\wedge \exists y$ and 
$\exists x\odot \forall(x\rs y)\le \exists x\odot (\exists x\rs \exists y)=\exists x\wedge \exists y$. \\ 
$(2)$ By $(M_6)$ and Proposition \ref{mpsBCK-10}$(4)$ we get 
$\forall((x\ra y)\odot (x\ra y))=\forall(x\ra y)\odot \forall(x\ra y)\le 
(\exists x\ra \exists y)\odot (\exists x\ra \exists y)$. 
Similarly, $\forall((x\rs y)\odot (x\rs y))\le (\exists x\rs \exists x)\odot (\exists x\rs \exists y)$. 
\end{proof}

The universal and existential quantifers on a pseudo MV-algebra $A$ were defined and investigated in \cite{Rac1} 
as follows. 

\begin{definition} \label{pm-psBE-10} $\rm($\cite{Rac1}$\rm)$
A universal quantifier on $A$ is a mapping $\forall:A\longrightarrow A$ satisfying the following conditions, 
for all $x, y\in A:$ \\
$(MVU_1)$ $\forall x\le x;$ \\
$(MVU_2)$ $\forall(x\wedge y)=\forall x\wedge \forall y;$ \\
$(MVU_3)$ $\forall((\forall x)^{-})=(\forall x)^{-}$, $\forall((\forall x)^{\sim})=(\forall x)^{\sim};$ \\ 
$(MVU_4)$ $\forall(\forall x\odot \forall y)=\forall x\odot \forall y;$ \\
$(MVU_5)$ $\forall(x \odot x)=\forall x\odot \forall x;$ \\
$(MVU_6)$ $\forall(x \oplus x)=\forall x\oplus \forall x$. 
\end{definition}

\begin{definition} \label{pm-psBE-20} $\rm($\cite{Rac1}$\rm)$
An existential quantifier on $A$ is a mapping $\exists:A\longrightarrow A$ satisfying the following conditions, 
for all $x, y\in A:$ \\
$(MVE_1)$ $x\le \exists x;$ \\
$(MVE_2)$ $\exists(x\vee y)=\exists x\vee \exists y;$ \\
$(MVE_3)$ $\exists((\exists x)^{-})=(\exists x)^{-}$, $\exists((\exists x)^{\sim})=(\exists x)^{\sim};$ \\ 
$(MVE_4)$ $\exists(\exists x\odot \exists y)=\exists x\odot \exists y;$ \\
$(MVE_5)$ $\exists(x \odot x)=\exists x\odot \exists x;$ \\
$(MVE_6)$ $\exists(x \oplus x)=\exists x\oplus \exists x$. 
\end{definition}

According to \cite{Ior1}, the bounded commutative pseudo BCK-algebras coincide (are categorically isomorphic) with pseudo MV-algebras, and so the bounded commutative pseudo BE-algebras have also this property. 
Namely, given a bounded commutative pseudo BE-algebra $(A,\ra,\rs,0,1)$, then the structure 
$(A,\odot,\oplus,^{-},^{\sim},0,1)$ defined by $x^{-}=x\ra 0$, $x^{\sim}=x\rs 0$, 
$x\odot y=(x\ra y^{-})^{\sim}=(y\rs x^{\sim})^{-}$, $x\oplus y=(y^{-}\odot x^{-})^{\sim}=(y^{\sim}\odot x^{\sim})^{-}$, is a pseudo MV-algebra. 
Conversely, given a pseudo MV-algebra $(A,\odot,\oplus,^{-},^{\sim},0,1)$, then the structure $(A,\ra,\rs,0,1)$ defined by $x\ra y=y\oplus x^{-}=(x\odot y^{\sim})^{-}$, $x\rs y=x^{\sim}\oplus y=(y^{-}\odot x)^{\sim}$ is a bounded 
commutative pseudo BE-algebra. \\
Moreover, all properties proved in this section for the case of monadic pseudo BCK(pP)-algebras and monadic bounded pseudo BCK-semilattices are also valid for monadic bounded commutative pseudo BE-algebras.   

\begin{proposition} \label{pm-psBE-30} Let $A$ be a bounded commutative pseudo BE-algebra. The universal quantifier 
$\forall$ defined in Theorem \ref{minv-psBE-200} is a universal quantifier on the coresponding pseudo MV-algebra. 
\end{proposition}
\begin{proof}
With the quantifiers $\forall$ and $\exists$ defined in Theorem \ref{minv-psBE-200}, $(A,\exists,\forall)$ is a 
monadic bounded commutative pseudo BE-algebra, so it satisfies the axioms $(M_1)$-$(M_7)$. 
Obviously, $(U_1)$ implies $\forall x\le x$, for all $x\in A$, that is $(MVU_1)$. 
Axiom $(MVU_2)$ follows from Proposition \ref{mpsBE-60}$(3)$, axiom $(MVU_3)$ from 
Proposition \ref{mpsBE-50}$(7)$ and axiom $(MVU_4)$ from Proposition \ref{mpp-BCK-20}$(3)$. 
Since axioms $(MVU_5)$ and $(MVU_6)$ are in fact $(M_6)$ and $(M_7)$, we conclude that $\forall$ is a universal  
quantifier on the pseudo MV-algebra $A$. 
\end{proof}

\begin{proposition} \label{pm-psBE-40} Let $A$ be a bounded commutative pseudo BE-algebra. 
The existential quantifier $\exists$ defined in Theorem \ref{minv-psBE-210} is an existential quantifier on the coresponding pseudo MV-algebra. 
\end{proposition}
\begin{proof}
It was proved in Theorem \ref{minv-psBE-210} that the axioms $(E_1)$-$(E_6)$ imply the axioms $(U_1)$-$(U_6)$ from 
Theorem \ref{minv-psBE-200}. Moreover, by Proposition \ref{pm-psBE-30}, axioms $(U_1)$-$(U_6)$ imply axioms $(MVU_1)$-$(MVU_6)$. 
Obviously, $(E_1)$ implies $(MVE_1)$, while by Proposition \ref{minv-psBE-70}$(1)$, $(MVU_2)$ is equivalent to $(MVE_2)$. Axiom $(MVE_3)$ follows from Proposition \ref{mpsBE-50}$(8)$, and axiom $(MVE_4)$ is a consequence of 
Proposition \ref{mpp-BCK-20}$(1)$. 
By Proposition \ref{minv-psBE-70}, axiom $(MVU_5)$ is equivalent to $(MVE_5)$, while $(MVU_6)$ is equivalent to $(MVE_6)$. It follows that $\exists$ is a existential quantifier on the pseudo MV-algebra $A$. 
\end{proof}

$\vspace*{5mm}$

\section{Monadic deductive systems of monadic pseudo BE-algebras}

The monadic deductive systems and monadic congruences of monadic pseudo BE-algebras are defined and their
properties are studied. It is proved that, in the case of a monadic distributive commutative
pseudo BE-algebra there is a one-to-one correspondence between monadic congruences and
monadic deductive systems, and it is swown that the quotient of a monadic pseudo BE-algebra via a monadic deductive system is a monadic BE-algebra. For the particular case of a monadic pseudo BCK-meet-semilattice, we prove that 
there is a one-to-one correspondence between monadic relative congruences and monadic normal deductive systems. 
We also show that the quotient of a monadic pseudo BCK-meet-semilattice via a monadic normal deductive system is a monadic pseudo BCK-meet-semilattice. \\

A subset $D$ of a pseudo BE-algebra $A$ is called a \emph{deductive system} of $A$ if it satisfies 
the axioms: 
$(ds_1)$ $1\in D,$ $(ds_2)$ for all $x, y\in A$, $x\in D$ and $x\ra y\in D$ imply $y\in D$. \\
A subset $D$ of $A$ is a deductive system if and only if it satisfies $(ds_1)$ and the axiom: \\
$(ds^{'}_2)$ for all $x, y\in A$, $x\in D$ and $x\rs y\in D$ imply $y\in D$. \\
Denote by ${\mathcal DS}(A)$ the set of all deductive systems of $A$. 
If $D\in {\mathcal DS}(A)$ and $x, y\in X$ such that $x\le y$, then $y\in D$ (\cite{Bor2}). 
A deductive system $D$ of $A$ is \emph{proper} if $D\ne A$. 
A deductive system $D$ of a pseudo BE-algebra $A$ is said to be \emph{normal} if it satisfies the condition:\\
$(ds_3)$ for all $x, y \in A$, $x \ra y \in D$ iff $x \rs y \in D$. \\
Denote by ${\mathcal DS_n}(A)$ the set of all normal deductive systems of $A$. 
If $A$ is a distributive pseudo BE-algebra, then ${\mathcal DS}(A)={\mathcal DS_n}(A)$ (\cite{Bor4}). 
If $X\subseteq A$, denote $[X)=\cap \{D\in {\mathcal DS}(A) \mid X\subseteq D\}$ the smallest deductive 
system of $A$ containing $X$ and it is called the deductive system \emph{generated} by $X$. If $X=\{x\}$ 
we write $[x)$ instead of $[\{x\})$.  
For details regarding deductive systems and congruence relations on a pseudo BE-algebra we refer the reader to  \cite{Bor2, Rez1}. 

\begin{definition} \label{mds-psBE-10} Let $(A,\exists,\forall)$ be a monadic pseudo BE-algebra and let 
$D\in {\mathcal DS}(A)$. Then $D$ is called a \emph{monadic deductive system} of $(A,\exists,\forall)$ 
if $x\in D$ implies $\forall x\in D$. 
\end{definition}

Denote by ${\mathcal MDS}(A,\exists,\forall)$ the set of all monadic deductive systems of $(A,\exists,\forall)$. 
Obviously, $\{1\}, A\in {\mathcal MDS}(A,\exists,\forall)$. 

\begin{examples} \label{mds-psBE-20}
$(1)$ Consider the pseudo BE-algebra $A$ and ists monadic operators from Example \ref{mpsBE-10-30}.  
We can see that ${\mathcal DS}(A)$=\{\{1\}, \{1,a,d\}, \{1,b,c\}, A\} and 
${\mathcal MDS}(A,\exists_1,\forall_1)={\mathcal MDS}(A,\exists_2,\forall_2)={\mathcal MDS}(A,\exists_3,\forall_3)=
{\mathcal MDS}(A,\exists_4,\forall_4)={\mathcal DS}(A)$. \\
$(2)$ In the Example \ref{minv-psBE-40-10} we have ${\mathcal DS}(A)=\{\{1\}, \{1,a\}, \{1,b\}, A\}$, 
${\mathcal MDS}(A,\exists_1,\forall_1)={\mathcal DS}(A)$ and 
${\mathcal MDS}(A,\exists_2,\forall_2)=\{\{1,\}, A\}$. 
\end{examples}

\begin{proposition} \label{mds-psBE-30} If $A$ is a pseudo BE(M)-algebra and $\emptyset \ne X\subseteq A$, then \\
$\hspace*{0.5cm}$  $[X)=\{x\in A \mid a_1\ra (a_2\ra \cdots \ra (a_n\ra x)\cdots )=1$, 
                            for some $n\ge 1$ and $a_1,\cdots, a_n\in X\}$ \\
$\hspace*{1.25cm}$ $   =\{x\in A \mid a_1\rs (a_2\ra\cdots \rs (a_n\rs x)\cdots )=1$, 
                            for some $n\ge 1$ and $a_1,\cdots, a_n\in X\}$.   
\end{proposition}
\begin{proof}

Similarly as \cite[Lemma 1.9]{Ciu2}, based on $(psBE_4)$ and property (M). 
\end{proof}

\begin{remark} \label{mds-psBE-30-10}
$(1)$ According to \cite[Th. 6]{Bor4}, any distributive pseudo BE-algebra satisfies property (M). 
With the assumption of distributivity, in \cite[Prop. 8]{Bor4} it is proved the same result as the one proved in Proposition \ref{mds-psBE-30}. \\
$(2)$ If $A$ is a commutative pseudo BE-algebra, then it is a commutative pseudo BCK-algebra, so that 
property (M) is satisfied. It follows that the above result is also valid in the case of commutative 
pseudo BE-algebras. 
\end{remark}

\begin{proposition} \label{mds-psBE-40} Let $(A,\exists,\forall)$ be a monadic pseudo BE(M)-algebra and let 
$D\in {\mathcal DS}(A)$. Then $D\in {\mathcal MDS}(A,\exists,\forall)$ if and only if $D=[D\cap A_{\exists\forall})$.
\end{proposition}
\begin{proof}
Let $D\in {\mathcal MDS}(A,\exists,\forall)$. Obviously, $[D\cap A_{\exists\forall})\subseteq D$.
If $x\in D$, then $\forall x\in D$. Since $\forall\forall x=\forall x$, we have $\forall x\in A_{\exists\forall}$, 
so $\forall x\in D\cap A_{\exists\forall}$. 
Moreover, $\forall x\ra x=1$, hence, by Proposition \ref{mds-psBE-30}, $x\in [D\cap A_{\exists\forall})$. 
It follows that $D\subseteq [D\cap A_{\exists\forall})$, thus $D=[D\cap A_{\exists\forall})$. 
Conversely, let $D\in {\mathcal DS}(A)$ such that $D=[D\cap A_{\exists\forall})$. 
According to Proposition \ref{mds-psBE-30}, for each $x\in D$, there exist $n\ge 1$ and 
$a_1,\cdots,a_n\in D\cap A_{\exists\forall}$ such that $a_1\ra (a_2\ra\cdots \ra (a_n\ra x)\cdots )=1$. 
Since $a_1,\cdots,a_n\in A_{\exists\forall}$, we have $a_1=\forall a_1=\exists a_1$, ... , 
$a_n=\forall a_n=\exists a_n$. By $(M_2)$ and $(M_4)$ we get: \\
$\hspace*{2.0cm}$ $1=\forall 1=\forall(a_1\ra (a_2\ra\cdots \ra (a_n\ra x)\cdots ))$ \\ 
$\hspace*{2.3cm}$ $=\forall(\exists a_1\ra (\exists a_2\ra\cdots \ra (\exists a_n\ra x)\cdots ))$ \\ 
$\hspace*{2.3cm}$ $=\exists a_1\ra \forall(\exists a_2\ra\cdots \ra (\exists a_n\ra x)\cdots )=\cdots$ \\
$\hspace*{2.3cm}$ $=\exists a_1\ra (\exists a_2\ra\cdots \ra (\exists a_n\ra \forall x)\cdots )$ \\
$\hspace*{2.3cm}$ $=a_1\ra (a_2\ra\cdots \ra (a_n\ra \forall x)\cdots )$. \\
Applying again Proposition \ref{mds-psBE-30}, it follows that $\forall x\in D$, so 
$D\in {\mathcal MDS}(A,\exists,\forall)$.   
\end{proof}

\begin{example} \label{mds-psBE-50} The pseudo BE-algebra $A$ from Example \ref{mpsBE-10-30} is distributive 
(\cite[Ex. 2]{Bor4}). Let $D=\{1,a,d\}\in {\mathcal DS}(A)$ from Example \ref{mpsBE-100}.   
Then we have $[D\cap A_{\exists_4\forall_4})=[\{1,d\})=\{1,a,d\}\in {\mathcal MDS}(A,\exists_4,\forall_4)$.
\end{example}

Let $A$ be a pseudo BE-algebra. An equivalence relation $\Theta$ on $A$ is called a congruence relation on $A$ if,  
for any $(x,y), (u,v)\in \Theta$ we have $(x\ra u,y\ra v), (x\rs u,y\rs v)\in \Theta$.  
If $a\in A$, then we denote $[a]_{\Theta}=\{x\in A \mid (x,a)\in \Theta\}$.  
Denote by ${\mathcal CON}(A)$ the set of all congruence relations on $A$. 

\begin{lemma} \label{mds-psBE-60} Let $A$ be a pseudo BE-algebra and let $\Theta \in {\mathcal CON}(A)$. 
Then the following hold, for all $x, y\in A:$ \\
$(1)$ $x\ra y, y\ra x\in [1]_{\Theta};$ \\
$(2)$ $x\rs y, y\rs x\in [1]_{\Theta};$ \\
$(3)$ $x\ra y\in [1]_{\Theta}$ iff $x\rs y\in [1]_{\Theta};$ \\
If $A$ is commutative, then: \\
$(4)$ $x\ra y, y\ra x\in [1]_{\Theta}$ implies $(x,y)\in \Theta;$ \\
$(5)$ $x\rs y, y\rs x\in [1]_{\Theta}$ implies $(x,y)\in \Theta$. 
\end{lemma}
\begin{proof}
$(1)$ Let $(x,y)\in \Theta$. Since $\Theta$ is reflexive, $(x,x), (y,y)\in \Theta$. 
From $(x,y), (y,y)\in \Theta$, we get $(x\ra y,y\ra y)=(x\ra y,1)\in \Theta$, that is $x\ra y\in [1]_{\Theta}$.  
Similarly, from $(y,x),(x,x)\in \Theta$, we have $y\ra x\in [1]_{\Theta}$. \\
$(2)$ Similarly as $(1)$. \\
$(3)$ Assume that $x\ra y\in [1]_{\Theta}$, that is $(x\ra y,1)\in \Theta$. Since $(y,y)\in \Theta$, then 
$((x\ra y)\rs y,1\rs y)=(x\ra y)\rs y,y)\in \Theta$. 
Now, from $(x,x)\in \Theta$ and $(x\ra y)\rs y,y)\in \Theta$ we have $(x\rs ((x\ra y)\rs y),x\rs y)\in \Theta$. 
Since by Proposition \ref{psBE-40}$(4)$, $x\rs ((x\ra y)\rs y)=1$, we get $(1,x\rs y)\in \Theta$, that is 
$x\rs y\in [1]_{\Theta}$. 
Similarly, $x\rs y\in [1]_{\Theta}$ implies $x\ra y\in [1]_{\Theta}$. \\
$(4)$ Let $x, y\in A$ such that $x\ra y, y\ra x\in [1]_{\Theta}$, that is $(x\ra y,1)\in \Theta$ and 
$(y\ra x,1)\in \Theta$. 
From $(x\ra y,1), (y,y)\in \Theta$ we get 
$((x\ra y)\rs y,1\rs y)=(x\Cup_1 y,y)\in \Theta$.   
From $(y\ra x,1), (x,x)\in \Theta$ we get 
$((y\ra x)\rs x,1\rs x)=(y\Cup_1 x,x)\in \Theta$.   
Since $x\Cup_1 y=y\Cup_1 x$, from $(x,y\Cup_1 x), (x\Cup_1 y,y)\in \Theta$, by transitivity we 
get $(x,y)\in \Theta$. \\
$(5)$ Similarly as $(4)$. 
\end{proof}

\begin{definition} \label{mds-psBE-70} Let $(A,\exists,\forall)$ be a monadic pseudo BE-algebra and let 
$\Theta \in {\mathcal CON}(A)$. Then $\Theta$ is called a \emph{monadic congruence relation} on $A$, 
if $(x,y)\in \Theta$ implies $(\forall x,\forall y)\in \Theta$, for all $x, y\in A$. 
\end{definition}

Denote by ${\mathcal MCON}(A,\exists,\forall)$ the set of all monadic congruence relations on $(A,\exists,\forall)$. 

\begin{theorem} \label{mds-psBE-80} Let $(A,\exists,\forall)$ be a monadic pseudo BE-algebra and let 
$\Theta \in {\mathcal MCON}(A,\exists,\forall)$. Then: \\
$(1)$ $[1]_{\Theta}\in {\mathcal MDS}(A,\exists,\forall);$ \\
$(2)$ if $A$ is commutative, then $(\exists x, \exists y)\in \Theta$, for all $(x,y)\in \Theta$. 
\end{theorem}
\begin{proof}
$(1)$ Obviuosly, $1\in {\mathcal MDS}(A,\exists,\forall)$. 
Let $x, y\in A$ such that $x, x\ra y\in {\mathcal MDS}(A,\exists,\forall)$, that is $(x,1), (x\ra y,1)\in \Theta$. 
From $(x,1)\in \Theta$ and $(y,y)\in \Theta$ we get $(x\ra y,1\ra y)=(x\ra y,y)\in \Theta$. 
Now, from $(1,x\ra y), (x\ra y,y)\in \Theta$, by transitivity, we have $(1,y)\in \Theta$, that is 
$y\in [1]_{\Theta}$. Hence $[1]_{\Theta}\in {\mathcal DS}(A)$. 
Let $x\in [1]_{\Theta}$, so $(x,1)\in \Theta$. Since $\Theta \in {\mathcal MCON}(A,\exists,\forall)$, we get  
$(\forall x,\forall 1)=(\forall x,1)\in \Theta$, that is $\forall x\in [1]_{\Theta}$. 
It follows that $[1]_{\Theta}\in {\mathcal MDS}(A,\exists,\forall)$. \\
$(2)$ Since $A$ is commutative, then $A$ is a pseudo BCK-algebra, so $A$ has property (M) and 
$x\Cup_1 y=y\Cup_1 x$ and $x\Cup_2 y=y\Cup_2 x$, for all $x, y\in A$. 
Let $(x,y)\in \Theta$. By Lemma \ref{mds-psBE-60}$(1)$ we have $x\ra y, y\ra x\in [1]_{\Theta}$. 
Since by $(1)$, $[1]_{\Theta}\in {\mathcal MDS}(A,\exists,\forall)$, we get 
$\forall(x\ra y),\forall(y\ra x)\in [1]_{\Theta}$. 
By $(M_1)$, $y\le \exists y$, and applying property $(M)$ we get $x\ra y\le x\ra \exists y$. 
Hence, by Corollary \ref{mpsBE-25-10} and $(M_3)$ we have 
$\forall(x\ra y)\le \forall(x\ra \exists x)=\exists x\ra \exists y$. 
Since $[1]_{\Theta}\in {\mathcal DS}(A)$, we get $\exists x\ra \exists y\in [1]_{\Theta}$.
Similarly, $\exists y\ra \exists x\in [1]_{\Theta}$. 
Applying Lemma \ref{mds-psBE-60}$(4)$, we conclude that $(\exists x, \exists y)\in \Theta$. 
\end{proof}

\begin{remark} \label{mds-psBE-80-10}
Let $A$ be a distributive pseudo BE-algebra. 
Given $D\in {\mathcal DS}(A)$, the relation $\Theta_D$ on $A$ defined by $(x,y)\in \Theta_D$ iff 
$x\ra y\in D$ and $y\ra x\in D$ is a congruence on $D$. 
We write $x/D=[x]_{\Theta_D}$ for every $x\in A$ and we have $D=[1]_{\Theta_D}$.  
Then $(A/\Theta_H,\ra,[1]_{\Theta_D})=(A/\Theta_D,\rs,[1]_{\Theta_D})$ is a 
BE-algebra called the \emph{quotient pseudo BE-algebra via $D$} and denoted by $A/D$ (\cite[Th. 4.2]{Rez1}). 
\end{remark}

\begin{theorem} \label{mds-psBE-90} Let $(A,\exists,\forall)$ be a monadic distributive commutative pseudo BE-algebra. 
Then there is a one-to-one correspondence between ${\mathcal MCON}(A,\exists,\forall)$ and 
${\mathcal MDS}(A,\exists,\forall)$. 
\end{theorem}
\begin{proof}
Since $A$ is commutative, $A$ is a pseudo BCK-algebra, hence it satisfies property (A). \\
If $\Theta \in {\mathcal MCON}(A,\exists,\forall)$, then by Theorem \ref{mds-psBE-80}, 
$[1]_{\Theta}\in {\mathcal MDS}(A,\exists,\forall)$.  
Conversely, let $D\in {\mathcal MDS}(A,\exists,\forall)$. Since $A$ is distributive, according to 
\cite[Prop. 4.1]{Rez1}, $\Theta_D\in {\mathcal CON}(A)$.  
Let $(x,y)\in \Theta_D$, so $x\ra y, y\ra x\in D$. 
Since $D\in {\mathcal MDS}(A,\exists,\forall)$, we have $\forall(x\ra y), \forall(y\ra x)\in D$.  
From $\forall x\le x$, using property (A) we get $x\ra y\le \forall x\ra y$, so by Corollary \ref{mpsBE-25-10} and $(M_4)$, $\forall(x\ra y)\le \forall(\forall x\ra y)=\forall(\exists\forall x\ra y)=\exists\forall x\ra \forall y=
\forall x\ra \forall y$. 
Hence, $\forall x\ra \forall y\in D=[1]_{\Theta_D}$ and similarly, $\forall y\ra \forall x\in D=[1]_{\Theta_D}$. 
Using Lemma \ref{mds-psBE-60}$(4)$, we get $(\forall x,\forall y)\in \Theta_D$, hence 
$\Theta_D \in {\mathcal MCON}(A,\exists,\forall)$. 
\end{proof}

\begin{theorem} \label{mds-psBE-110} Let $(A,\exists,\forall)$ be a monadic distributive commutative pseudo BE-algebra 
and let $D\in {\mathcal MDS}(A,\exists,\forall)$. 
Then $(A/D, \exists_D, \forall_D)$ is a monadic BE-algebra, where $\exists_D x=(\exists x)/D$ and 
$\forall_D x=(\forall x)/D$. 
\end{theorem}
\begin{proof}
It follows by Remark \ref{mds-psBE-80-10} and Theorems \ref{mds-psBE-80}, \ref{mds-psBE-90}. 
\end{proof}

\begin{remark} \label{mds-psBE-110-10} Let $(A,\ra,\rs,1)$ be a pseudo BCK-algebra. \\
$(1)$ Let $H\in {\mathcal DS}(A)$. 
Then $H\in {\mathcal DS}_n(A)$ if and only if the relation $\Theta_H$ defined by \\
$\hspace*{3cm}$ $(x,y)\in \Theta_H$ iff $x\ra y\in H$ and $y\ra x\in H$ \\
is a congruence on $A$. In this case, $[1]_{\Theta_H}=H$ (\cite[Prop. 2.2.2]{Kuhr6}). \\
$(2)$ We say that $\Theta\in {\mathcal CON}(A)$ is a \emph{relative congruence} on $A$ if the quotient algebra 
$(A/{\Theta},\ra,\rs,[1]_{\Theta})$ is a pseudo BCK-algebra. 
Denote by ${\mathcal CON}_r(A)$ the set of all relative congruences on $A$. 
There is a one-to-one correspondence between ${\mathcal CON}_r(A)$ and ${\mathcal DS}_n(A)$ 
(\cite[Prop. 2.2.4]{Kuhr6}). \\
$(3)$ Let $(A,\wedge,\ra,\rs,1)$ be a pseudo BCK-meet-semilattice. An equivalence relation $\Theta$ on $A$ is 
called a congruence relation on $A$ if, for any $(x,y), (u,v)\in \Theta$ we have 
$(x\ra u,y\ra v), (x\rs u,y\rs v), (x\wedge u,y\wedge v)\in \Theta$.  
Since any pseudo BCK-meet-semilattice is a pseudo BE-algebra, then the results proved in 
Lemma \ref{mds-psBE-60}$(1)$-$(3)$ and Theorem \ref{mds-psBE-80}$(1)$ for pseudo BE-algebras are also valid 
for the case of pseudo BCK-meet-semilattices. 
\end{remark}

\begin{proposition} \label{mds-psBE-120} Let $A$ be a pseudo BCK-meet-semilattice and let $\Theta \in {\mathcal CON}(A)$. 
The following hold, for all $(x,y)\in A:$ \\
$(1)$ $x\ra y, y\ra x\in [1]_{\Theta}$ implies $(x,y)\in \Theta;$ \\
$(2)$ $x\rs y, y\rs x\in [1]_{\Theta}$ implies $(x,y)\in \Theta$.  
\end{proposition}
\begin{proof}
$(1)$ From $(x\ra y,1), (y,y)\in \Theta$ we get $((x\ra y)\rs y,1\rs y)\in \Theta$, that is 
$((x\ra y)\rs y,y)\in \Theta$. 
Moreover, $((x\ra y)\rs y,y), (x,x)\in \Theta$ implies $(((x\ra y)\rs y)\wedge x,y\wedge x)\in \Theta$, hence, 
by Proposition \ref{psBE-20}$(5)$, we get $(x,y\wedge x)\in \Theta$. 
Similarly, $(y\ra x,1), (x,x)\in \Theta$ implies $((y\ra x)\rs x,1\rs x)\in \Theta$, so 
$((y\ra x)\rs x,x)\in \Theta$.
From $((y\ra x)\rs x,x), (x,x)\in \Theta$, we have  $(((y\ra x)\rs x)\wedge y, x\wedge y)\in \Theta$, so 
$(y,x\wedge y)\in \Theta$. 
Finally, from $(x,x\wedge y), (x\wedge y, y)\in \Theta$, we get $(x,y)\in \Theta$. \\
$(2)$ Similarly as $(1)$. 
\end{proof}

\begin{proposition} \label{mds-psBE-120-10} Let $(A,\exists,\forall)$ be a monadic pseudo BCK-meet-semilattice and 
let $\Theta \in {\mathcal MCON}(A,\exists,\forall)$. Then $(\exists x, \exists y)\in \Theta$. 
\end{proposition}
\begin{proof}
Let $(x,y)\in \Theta$. By Lemma \ref{mds-psBE-60}$(1)$ we have $x\ra y, y\ra x\in [1]_{\Theta}$. 
Since by Theorem \ref{mds-psBE-80}$(1)$, $[1]_{\Theta}\in {\mathcal MDS}(A,\exists,\forall)$, we get 
$\forall(x\ra y),\forall(y\ra x)\in [1]_{\Theta}$. 
By $(M_1)$, $y\le \exists y$, so $x\ra y\le x\ra \exists y$. 
Hence, by Corollary \ref{mpsBE-25-10} and $(M_3)$ we have 
$\forall(x\ra y)\le \forall(x\ra \exists x)=\exists x\ra \exists y$. 
Since $[1]_{\Theta}\in {\mathcal DS}(A)$, we get $\exists x\ra \exists y\in [1]_{\Theta}$.
Similarly, $\exists y\ra \exists x\in [1]_{\Theta}$. 
Applying Proposition \ref{mds-psBE-120}, we conclude that $(\exists x, \exists y)\in \Theta$.
\end{proof}

\begin{theorem} \label{mds-psBE-130} Let $(A,\exists,\forall)$ be a monadic pseudo BCK-meet-semilattice.
Then there is a one-to-one correspondence between ${\mathcal MCON}_r(A,\exists,\forall)$ and 
${\mathcal MDS}_n(A,\exists,\forall)$. 
\end{theorem}
\begin{proof} 
If $\Theta \in {\mathcal MCON}_r(A,\exists,\forall)$, then by Theorem \ref {mds-psBE-80} and 
Remark \ref{mds-psBE-110-10}, $[1]_{\Theta}\in {\mathcal MDS}_n(A,\exists,\forall)$.  
Conversely, let $D\in {\mathcal MDS}_n(A,\exists,\forall)$. Applying again Remark \ref{mds-psBE-110-10}, 
$\Theta_D\in {\mathcal CON}_r(A)$.  
Let $(x,y)\in \Theta_D$, so $x\ra y, y\ra x\in D$. 
Since $D\in {\mathcal MDS}_n(A,\exists,\forall)$, we have $\forall(x\ra y), \forall(y\ra x)\in D$.  
From $\forall x\le x$, we get $x\ra y\le \forall x\ra y$, so by Corollary \ref{mpsBE-25-10}, $(M_4)$ and $(M_5)$, 
$\forall(x\ra y)\le \forall(\forall x\ra y)=\forall(\exists\forall x\ra y)=\exists\forall x\ra \forall y=
\forall x\ra \forall y$. 
Hence, $\forall x\ra \forall y\in D=[1]_{\Theta_D}$ and similarly, $\forall y\ra \forall x\in D=[1]_{\Theta_D}$. 
Using Proposition \ref{mds-psBE-120}, we get $(\forall x,\forall y)\in \Theta_D$, hence 
$\Theta_D \in {\mathcal MCON}_r(A,\exists,\forall)$. 
\end{proof}

\begin{theorem} \label{mds-psBE-140} Let $(A,\exists,\forall)$ be a monadic pseudo BCK-meet-semilattice  
and let $D\in {\mathcal MDS}_n(A,\exists,\forall)$. 
Then $(A/D, \exists_D, \forall_D)$ is a monadic pseudo BCK-meet-semilattice, where $\exists_D x=(\exists x)/D$ and 
$\forall_D x=(\forall x)/D$. 
\end{theorem}
\begin{proof}
It follows by Remark \ref{mds-psBE-110-10}, Proposition \ref{mds-psBE-120-10} and Theorem \ref{mds-psBE-130}. 
\end{proof}




$\vspace*{5mm}$

\vspace*{3mm}

\end{document}